\documentclass[11pt,reqno]{amsproc}
\usepackage[margin=1in]{geometry}
\usepackage{amsmath, amsthm, amssymb}
\usepackage[colorlinks=true, pdfstartview=FitV, linkcolor=blue,citecolor=blue, urlcolor=blue]{hyperref}
\usepackage[abbrev,lite,nobysame]{amsrefs}
\usepackage{times}
\usepackage[usenames,dvipsnames]{color}

\usepackage{mathtools}

\mathtoolsset{showonlyrefs=true}
\usepackage{dsfont}

\providecommand{\ip}[1]{\langle#1\rangle}
\providecommand{\abs}[1]{\left\lvert#1\right\rvert}
\providecommand{\norm}[1]{\left\|#1\right\|}

\newcommand{\N}{\mathbb{N}}
\newcommand{\Z}{\mathbb{Z}}
\newcommand{\Lin}{\mathcal{L}}
\newcommand{\PP}{\mathbb{P}_0}
\newcommand{\NP}{\mathbb{P}_{\neq}}
\newcommand{\Pk}{\mathbb{P}_{k}}

\def\eps{\varepsilon}
\def\e{{\rm e}}
\def\dd{{\rm d}}

\def\ddt{{\frac{\dd}{\dd t}}}
\def\R {\mathbb{R}}

\def \l {\langle}
\def \r {\rangle}
\def\T {{\mathbb T}}
\def\de{{\partial}}

\newtheorem{proposition}{Proposition}[section]
\newtheorem{theorem}[proposition]{Theorem}
\newtheorem{corollary}[proposition]{Corollary}
\newtheorem{lemma}[proposition]{Lemma}
\newtheorem*{bootstrap*}{Bootstrap Step}
\theoremstyle{definition}
\newtheorem{definition}[proposition]{Definition}
\newtheorem{remark}[proposition]{Remark}

\numberwithin{equation}{section}

\title[Navier-Stokes equations near the Poiseuille flow]{Enhanced dissipation in the Navier-Stokes equations near the Poiseuille flow}

\author[M.\ Coti Zelati]{Michele Coti Zelati}
\address{Department of Mathematics, Imperial College London, London, SW7 2AZ, UK}
\email{m.coti-zelati@imperial.ac.uk}
\author[T.M.\ Elgindi]{Tarek M. Elgindi}
\address{Department of Mathematics, UC San Diego, La Jolla, CA 92093, USA}
\email{telgindi@ucsd.edu}
\author[K.\ Widmayer]{Klaus Widmayer}
\address{Institute of Mathematics, EPFL, Station 8, 1015 Lausanne, Switzerland}
\email{klaus.widmayer@epfl.ch}



\begin{document}

\begin{abstract}
We consider solutions to the 2d Navier-Stokes equations on $\mathbb{T}\times\mathbb{R}$ close to the Poiseuille flow, with small viscosity $\nu>0$. Our first result concerns a semigroup estimate for the linearized problem. Here we show that the $x$-dependent modes of linear solutions decay on a time-scale proportional to $\nu^{-1/2}|\log\nu|$. This effect is often referred to as \emph{enhanced dissipation} or \emph{metastability} since it gives a much faster decay than the regular dissipative time-scale $\nu^{-1}$ (this is also the time-scale on which the $x$-independent mode naturally decays). We achieve this using an adaptation of the method of hypocoercivity. 

Our second result concerns the full nonlinear equations. We show that when the perturbation from the Poiseuille flow is initially of size at most $\nu^{3/4+}$, then it remains so for all time. Moreover, the enhanced dissipation also persists in this scenario, so that the $x$-dependent modes of the solution are dissipated on a time scale of order $\nu^{-1/2}|\log\nu|$. This transition threshold is established by a bootstrap argument using the semigroup estimate and a careful analysis of the nonlinear term in order to deal with the unboundedness of the domain and the Poiseuille flow itself. 
\end{abstract}

\vspace*{2.2cm}

\maketitle

\tableofcontents

\section{Introduction}
In this paper we consider the two-dimensional Navier-Stokes equations
\begin{align}\label{eq:NSEvel}
\begin{cases}
\de_t U+(U\cdot \nabla) U+\nabla P-\nu\Delta U=0,\\
\nabla\cdot U=0,
\end{cases}
\end{align}
posed on the boundary-less domain $\T\times\R$, where $\T=[0,2\pi)$ is a periodic interval. Here,  $U=(U^1,U^2)$ is the
velocity vector field and $P$ is the scalar pressure of an incompressible fluid of uniform density, and the kinematic viscosity 
$\nu>0$ is proportional to the inverse of the Reynolds number. Setting $\nabla^\perp=(-\de_y,\de_x)$ the rotation of the gradient vector, the
vorticity $\Omega:=\nabla^\perp\cdot U=-\de_y U^1+\de_xU^2$ satisfies the active scalar equation
\begin{align}\label{eq:NSEvort}
\begin{cases}
\de_t \Omega+U\cdot \nabla \Omega-\nu\Delta \Omega=0,\\
U=\nabla^\perp \Psi, \quad \Delta\Psi=\Omega,
\end{cases}
\end{align}
where $\Psi$ is the corresponding stream-function.

A widely studied stationary solution to \eqref{eq:NSEvel} is the so-called
\emph{Poiseuille flow}, given by $U^{S}(x,y)=(y^2,0)$ and $P^{S}(x,y)=2\nu x$. There are several reasons for which this flow
is of basic importance: On the one hand, it is the two-dimensional version of the three-dimensional pipe flow studied by Reynolds in his famous experiments 
\cite{Reynolds83}, the subtleties of which are yet to be understood  \cite{WPKM08}. It is also the prototypical example of a strictly convex shear flow, whose stability properties have been widely studied in the physics literature since Rayleigh \cites{rayleigh1879, DrazinReid81}. Moreover, it is the simplest non-trivial example of a shear flow on $\T\times\R$ besides the Couette flow $(y,0)$.

The main goal of this article is to prove quantitative stability results for the Poiseuille flow. By writing $U=(y^2,0)+u$, with corresponding $\Omega=-2y+\omega$, we can rewrite \eqref{eq:NSEvort} as
\begin{align}\label{eq:NSEpois}
\begin{cases}
\de_t \omega+y^2\de_x\omega -2\de_x\psi-\nu\Delta \omega=-u\cdot \nabla \omega,\\
u=\nabla^\perp \psi, \quad \Delta\psi=\omega.
\end{cases}
\end{align}
Here $u:\R\times\T\times\R\to\T\times\R$ and $\omega,\psi:\R\times\T\times\R\to\R$ are thought of as perturbations of the velocity, vorticity and stream-function around the Poiseuille flow. 

This way the question of stability of the Poiseuille flow can be rephrased in terms of the behavior of (small data) solutions to \eqref{eq:NSEpois}. Our analysis is divided into two parts: Firstly, in Section \ref{sec:LinEn} we investigate the decay properties of the linear part of \eqref{eq:NSEpois}. Here we will see that solutions decay at a faster than dissipative time scale -- an effect referred to as \emph{enhanced dissipation}. Secondly, Section \ref{sec:nonlin} shows that there is a threshold for the size of the initial data in a suitable norm, below which this fast decay also persists in the nonlinear case. Here the size of the threshold is given in terms of the viscosity.
In the remainder of this introduction we discuss our results in more detail, and give an overview of the paper.

\subsection{Linear enhanced dissipation}
In our first main result, we study the decay properties
of the semigroup generated by the linear operator 
\begin{align}
 \Lin=-y^2\de_x+2\de_x\Delta^{-1}+\nu\Delta,
\end{align}
in the  weighted $L^2$ space $X$ normed by
\begin{align}\label{eq:normX}
\| f\|^2_X:=\| f\|^2_{L^2}+\| y f\|^2_{L^2}.
\end{align}
We will prove the following theorem.
\begin{theorem}\label{thm:linenhanced}
Let $\nu<1$ and let $\Pk$ denote the projection to the sum of the $\pm k$-th Fourier modes in $x$, $k\in\N$. Then there exist $\eps_0\geq \frac{1}{20}$ and $C_0\in(1,10)$
such that for every $k\neq 0$ there holds the semigroup estimate
\begin{align}
\| \e^{\Lin t}\Pk \|_{X\to X}\leq C_0  \e^{-\eps_0\lambda_{\nu,k} t},  \qquad \forall t\geq 0,
\end{align}
where
\begin{align}
\lambda_{\nu,k}=\frac{\nu^{1/2}|k|^{1/2}}{1+|\log \nu|+\log |k|}.
\end{align}
The constants $\eps_0$ and $C_0$ can be explicitly computed.
\end{theorem}
By summing over all $k\in\N$, we can re-state the above result in the following unified way. Here, we need to require that the zeroth $x$-frequency (i.e. the
$x$-average of the solution) vanishes. Since the projections $\Pk$ and $\Lin$ commute, this property is preserved by the flow generated by the semigroup.
\begin{corollary}\label{cor:linenhanced}
Let $\nu<1$ and $\omega^{in} \in X$, and assume that for almost every $y\in \R$ we have
\begin{align}\label{eq:xmeanfree}
\int_\T \omega^{in}(x,y)\dd x=0.
\end{align}
Then
\begin{align}\label{eq:lindecay}
\| \e^{\Lin t} \omega^{in}\|_{X}\leq C_0  \e^{-\eps_0\lambda_{\nu} t}\|\omega^{in}\|_{X},  \qquad \forall t\geq 0,
\end{align}
where
\begin{align}\label{eq:rate}
\lambda_{\nu}=\frac{\nu^{1/2}}{1+|\log \nu|}.
\end{align}
\end{corollary}
Estimate \eqref{eq:lindecay} is a quantitative account of linear \emph{enhanced dissipation}: initial data in $X$ satisfying \eqref{eq:xmeanfree} 
decay exponentially at a rate proportional to $\nu^{1/2}$ (up to a logarithmic correction), which is much faster than the heat equation rate proportional to $\nu$.
This is a manifestation of metastability: the nonzero $x$-frequencies decay at the fast time-scale $O(\lambda_\nu^{-1})$, while the $x$-independent modes have a much longer relaxation time proportional to  $O(\nu^{-1})$.

The phenomenon of enhanced dissipation has been widely studied in the physics literature \cites{BL94,LB01,Lundgren82,rhines1983rapidly}, 
and has recently received a lot of 
attention from the mathematical community: from the seminal article \cite{CKRZ08}, quantitative questions have been addressed in the context of
passive scalars \cites{CZDE18, BCZGH15, BCZ15, BW13, WEI18} and Navier-Stokes equations near the Couette flow \cites{BMV14, WZ18, GNRS18}, as well as Lamb-Oseen vortices \cites{Gallay2017, LiWeiZhang2017}. However, the only other quantitative result on the
Navier-Stokes equations near a shear flow (analogous to Theorem \ref{thm:linenhanced}) has been obtained in \cites{WZZkolmo17, IMM17} for the Kolmogorov flow $(\sin y,0)$. In these works, the analysis required a careful and complicated study of the spectral properties of the linear operator $\Lin$ and its resolvent. 

Our result is the first of its kind for the Navier-Stokes equations on $\T\times \R$ (excluding the explicitly solvable Couette flow, already analyzed by Lord Kelvin in 1887 \cite{Kelvin87}). In stark contrast to the aforementioned case of the Kolmogorov flow, our result is proved via a simple energy method based on a modification of the so-called hypocoercivity framework \cite{Villani09}.
The pertinent ideas are discussed in detail in Section \ref{sec:LinEn}: We present first our claim in terms of a hypocoercivity functional (Theorem \ref{thm:phi_k} in Section \ref{ssec:main_hypo}), then give an overview of the energy estimates that enter into it (Section \ref{ssec:prelim_enest}). Subsequently, Section \ref{ssec:hypoco} establishes its proof. To conclude, in Section \ref{sub:semi} we deduce Corollary \ref{cor:fast_decay}, which implies Theorem \ref{thm:linenhanced}. We remark that in all of these considerations one can treat the $x$ frequencies separately, since they are naturally decoupled in the linear equations.

\subsection{Nonlinear transition stability threshold}
Our second result is concerned with the \emph{nonlinear} asymptotic stability of the Poiseuille flow. This question is also related to what is known as \emph{subcritical transition}: for the present scenario, it was conjectured by Lord Kelvin \cite{Kelvin87} that indeed the flow is stable, but that the stability threshold decreases as $\nu\to0$. For any real system this then entails transition at a finite Reynolds number. 
Our result confirms this behavior rigorously for the Poiseuille flow. We give an explicit size in terms of powers of the viscosity $\nu$, such that initial data below this threshold  (in a suitable norm) yield global solutions \emph{that exhibit enhanced dissipation}.
More precisely, our result is as follows.
\begin{theorem}\label{thm:threshold-main}
 For every $\mu,C_1>0$, there exists $\nu_0>0$ such that if $0<\nu\leq\nu_0$ and 
 \begin{equation}\label{eq:trans}
  \norm{\omega^{in}}_{X} + \norm{y \PP u^{1,in}}_{L^2}\leq C_1\nu^{3/4+2\mu},
 \end{equation}
 then there exists a unique global solution $\omega$ to the initial value problem for \eqref{eq:NSEpois} with $\omega(0)=\omega^{in}$. 
 Moreover, the modes with $k\neq 0$ exhibit an enhanced dissipation rate as in the linear case, namely
 \begin{equation}
  \norm{\NP\omega(t)}_{X}\leq 2C_0\e^{-\frac{\eps_0}{1+\log2C_0}\lambda_\nu t}\norm{\NP\omega^{in}}_{X},\qquad \forall t\geq 0,
 \end{equation}
 where $\NP=1-\PP$, while the $k=0$ mode remains bounded, i.e.\ for a universal constant $C_2>0$ we have
 \begin{equation}
  \norm{\PP\omega(t)}_{X}\leq C_2\norm{\omega^{in}}_{X},\qquad \forall t\geq 0.
 \end{equation}
\end{theorem}

The basic setting for the proof of this theorem is a perturbative analysis based on the properties of the linear flow as studied in Section \ref{sec:LinEn}. At its heart is the idea to bootstrap global control over the $X$ norm, which we used in order to obtain the linear enhanced dissipation. We thus start by discussing basic properties and well-posedness of the full, nonlinear equation in the $X$ norm (see Section \ref{ssec:basics}). Since the arguments for the bootstrap itself vitally require a fine understanding of the nonlinear structure of the equations, we then give an account of the relevant details (Section \ref{ssec:struct}). Subsequently, Section \ref{ssec:prelim} illustrates how control of the $X$ norm yields some crucial energy estimates that will be used later. Finally, the proof of the claim via bootstrap is carried out in Section \ref{ssec:thresh-proof}.

The above Theorem \ref{thm:threshold-main} is the first such result for the Poiseuille flow, and it is worth highlighting again its relatively straightforward proof. To the best of our knowledge, outside the realm of monotone shear flows there is only one similar result: the article \cite{WZZkolmo17} establishes an analogous claim for the Kolmogorov flow (see also \cite{LWZ18} for results in three-dimensions). As in our paper, the nonlinear result therein is obtained by means of a perturbative analysis. This relies on the linear estimates established via spectral methods, which are significantly more complicated than our energy estimates, as mentioned above. In addition, the nonlinear analysis makes crucial use of an additional effect called inviscid damping. While this is in principle also present in our case, it plays no role in our analysis. This is a clear advantage, since at present it seems that inviscid damping for the Poiseuille flow is still not well understood in the unbounded domain $\T\times \R$. In contrast to \cite{WZZkolmo17}, one 
additional difficulty that we face in our analysis is the presence of the weight in the norm $X$, since our domain and the gradient of the Poiseuille flow are unbounded as $\abs{y}\to\infty$: this requires a finer analysis of the nonlinear terms in order to close the desired energy estimate. 

In the case of monotone shear flows, lately the Couette flow has attracted a lot of attention. We mention \cites{BVW16, CLWZ18, BGM15I, BGM15II, BGM15III, Liss18} and the
recent survey \cite{BGM18} for a variety of results with data in Sobolev and Gevrey spaces in both two and three space dimensions.

\begin{remark}
It is not clear if the decay rate $\lambda_\nu$ in \eqref{eq:rate} is sharp, although it coincides with the passive scalar rate in \cite{BCZ15}. The logarithmic correction
is likely a technical matter. Similarly, the exponent 3/4 in the transition threshold \eqref{eq:trans} is also unclear to be sharp; in fact, it may depend on the choice
of the norm, for which there is no clear physical justification. Our norm $X$ arises as a natural energy of the system.
\end{remark}

\section{Linear enhanced dissipation}\label{sec:LinEn}
This section is devoted to the proof of Theorem \ref{thm:linenhanced}. Setting  $\omega(t)=\e^{\Lin t} \omega^{in}$, we have that
\begin{align}\label{eq:poison}
\begin{cases}
\de_t\omega +y^2\de_x \omega -2\de_x\psi=\nu\Delta \omega,\\
\Delta\psi=\omega,
\end{cases}
\end{align}
with initial datum $\omega(0)=\omega^{in}$. Via an expansion of $\omega$ (and $\psi$) as a Fourier series in the $x$ variable, namely
\begin{align}
\omega(t,x,y)=\sum_{k\in \Z} a_k(t,y)\e^{ikx}, \qquad a_k(t,y)=\frac{1}{2\pi}\int_\T\omega(t,x,y)\e^{-ikx}\dd x,
\end{align} 
for $k\in\N_0$ we set
\begin{equation}
 \omega_k(t,x,y):=\sum_{\abs{l}=k} a_k(t,y)\e^{ikx}.
\end{equation}
This way we may express $\omega=\sum_{k\in\N_0}\omega_k(t,x,y)$ as a sum of \emph{real-valued} functions $\omega_k$ that are localized in $x$-frequency on a single band $\pm k$, $k\in\N_0$. We thus see that \eqref{eq:poison} decouples in $k$ and becomes an infinite system of one-dimensional equations.

\subsection*{Notation conventions for Section \ref{sec:LinEn}}
In what follows, we will use $\|\cdot\|$ and $\l\cdot,\cdot\r$ for the standard real $L^2$ norm and scalar product, respectively. 
We will not distinguish between one and two dimensional $L^2$ spaces, as no dimensional property will be used.

\subsection{The main result}\label{ssec:main_hypo}

The main result of this section is a decay estimate on $\omega_k$ for $k\in\N$. 

\begin{theorem}\label{thm:phi_k}
 Let $\nu>0$ and $k\in\N$. Then there exists $\eps_0\geq \frac{1}{20}$ such that the following holds true: there exist constants $\alpha_0,\beta_0,\gamma_0>0$ only depending on $\eps_0$ such that the energy functional
 \begin{equation}\label{eq:def_Phi_k}
  \Phi_k:=\frac12\left[\|\omega_k\|^2 + \frac{\alpha_0\nu^{1/2}}{\abs{k}^{1/2}} \|\nabla\omega_k\|^2+\frac{4\beta_0}{|k|} \l y\de_x \omega_k,\de_y\omega_k\r +\frac{\gamma_0}{\nu^{1/2}\abs{k}^{3/2}}\left(\|y\de_x \omega_k\|^2+ 2 \|\nabla \de_x\psi_k\|^2\right) \right]
 \end{equation}
 satisfies the differential inequality
  \begin{align}\label{eq:diffineq_final}
  \ddt \Phi_k +2\eps_0\nu^{1/2}\abs{k}^{1/2}\Phi_k+\frac{\alpha_0\nu^{3/2}}{2\abs{k}^{1/2}}\norm{\Delta\omega_k}^2+\frac{\gamma_0 \nu^{1/2}}{\abs{k}^{3/2}} \left[ \frac{7}{8}\norm{y \de_{x}\nabla\omega_k}^2+\|\de_x \omega_k \|^2\right]\leq 0
\end{align}
for all $t\geq 0$.
In particular, assuming that $\nabla\omega_k^{in},y\de_x\omega_k^{in}\in L^2$ we have
\begin{align}
 \Phi_k(t)\leq \e^{-2\eps_0\nu^{1/2}\abs{k}^{1/2}t}\Phi_k(0),\quad t\geq0.
\end{align}
\end{theorem}
The above decay estimate is obtained via a hypocoercivity argument that requires a weighted $H^1$-norm of the initial data to be finite.
However, as a direct consequence we obtain the following result.
\begin{corollary}\label{cor:fast_decay}
 With the same notation as in Theorem \ref{thm:phi_k}, let $C_0:=(3\e(1+2\alpha_0\eps_0))^{1/2}$, and assume that $0<\nu\leq 1$. Then
 \begin{equation}
  Q_k(t):=\frac{1}{2}\norm{\omega_k(t)}^2 +\frac{\gamma_0}{4}\left[\norm{y \omega_k(t)}^2 + 2\norm{\nabla \psi_k(t)}^2\right]
 \end{equation}
decays at the fast rate
\begin{equation}
 Q_k(t)\leq C_0^2 Q_k(0)\e^{-2\eps_0\lambda_{\nu,k} t}, \quad t\geq 0.
\end{equation}
where
\begin{align}
\lambda_{\nu,k}=\frac{\nu^{1/2}|k|^{1/2}}{1+|\log \nu|+\log |k|}.
\end{align}
\end{corollary}
Since $\norm{\nabla \psi_k(t)}\leq \norm{\omega_k(t)}$ for any $k\neq 0$, Corollary \ref{cor:fast_decay} implies Theorem \ref{thm:linenhanced}. 
Its proof is given in Section \ref{sub:semi}.

\begin{remark}
 Theorem \ref{thm:phi_k} has a corresponding version in case of partial dissipation $\nu\de_{yy}\omega$ (rather than the full Laplacian $\nu\Delta\omega$) on the right hand side of \eqref{eq:poison}. However, this does not translate to faster decay in $L^2$ as in Corollary \ref{cor:fast_decay}. Indeed, a key ingredient for its proof in the case of full Laplacian dissipation is the monotonicity of $\norm{y\de_x\omega}^2+2\norm{\nabla\de_x\psi}^2$, see also Remark \ref{rem:mon}. This, however, is no longer true for partial dissipation.
\end{remark}

\subsection{Preliminary energy estimates}\label{ssec:prelim_enest}
We start the discussion with some energy estimates that will be used to build the functionals $\Phi_k$. 

\begin{lemma}\label{lem:en_est}
 Let $\omega$ solve \eqref{eq:poison}. Then we have the following balances:
 \begin{align}
  \frac12\ddt \|\omega\|^2 +\nu\|\nabla\omega\|^2&=0,\label{eq:L2balance} \\
  \frac12\ddt \|\nabla\omega\|^2 +\nu\|\Delta\omega\|^2+2\l y\de_x \omega,\de_y\omega\r&=0,\label{eq:H1balance} \\
  \ddt\l y\de_x \omega,\de_y\omega\r+2\| y\de_x \omega\|^2+4\|\de_{xy}\psi\|^2 &=-2\nu\l \Delta \omega, y\de_{xy}\omega\r,\label{eq:crossterm}\\
  \frac12\ddt\| y\de_x \omega\|^2+\nu\|y \de_{x}\nabla\omega\|^2&=\nu\|\de_x\omega\|^2-4\l y\de_{xy} \psi, \de_{xx} \psi\r,\label{eq:gamma-weight}\\
  \frac12\ddt\| \nabla \de_x\psi\|^2 +\nu\|\de_x\omega\|^2&=2\l y\de_{xy} \psi,\de_{xx}\psi\r\label{eq:gamma-psi}.
 \end{align}
\end{lemma}

\begin{remark}\label{rem:mon}
 In particular, combining \eqref{eq:gamma-weight} and \eqref{eq:gamma-psi} we have the useful identity 
 \begin{equation}\label{eq:mon}
  \frac{1}{2}\ddt \left[\norm{y\de_x \omega}^2+2\norm{\nabla \de_x\psi}^2\right]=-\nu\norm{\de_x\omega}^2-\nu\norm{y \de_{x}\nabla\omega}^2.
 \end{equation}
 This will be used in the proof of Corollary \ref{cor:fast_decay} and also motivates the structure of the $\gamma$ term in our definition of $\Phi_k$.
 The cancellation obtained in this linear combination is in fact a crucial point in our argument.
\end{remark}

\begin{proof}
All estimates follow by direct computation, using integration by parts and the antisymmetry property $\ip{y^n\de_x f,f}=0$, for $n\in\N_0$.
The $L^2$ balance \eqref{eq:L2balance} follows directly by testing \eqref{eq:poison} with $\omega$:
\begin{align}
 \frac12\ddt \|\omega\|^2 +\nu\|\nabla\omega\|^2=0.
\end{align}
 Testing  \eqref{eq:poison} with $\Delta\omega$ we also obtain \eqref{eq:H1balance} by a simple integration by parts as
\begin{align}
 \frac12\ddt \|\nabla\omega\|^2 +\nu\|\Delta\omega\|^2=-2\l y\de_x \omega,\de_y\omega\r.
\end{align}
Turning to \eqref{eq:crossterm}, we use \eqref{eq:poison} to compute
\begin{align}
\ddt\l y\de_x \omega,\de_y\omega\r
&=\nu\left[\l y \de_{x}\Delta \omega,\de_y\omega\r+\l y\de_x \omega,\de_{y}\Delta \omega\r\right]-\left[\l y^3\de_{xx} \omega ,\de_y\omega\r+\l y\de_x \omega,\de_y(y^2\de_x \omega )\r\right]\notag\\
&\quad+2\left[\l y\de_{xx}\psi,\de_y\omega\r+\l y\de_x \omega,\de_{xy}\psi\r\right].
\end{align}
We treat the $\nu$ term integrating by parts as
\begin{equation}
\l y \de_{x}\Delta \omega,\de_y\omega\r+\l y\de_x \omega,\de_{y}\Delta \omega\r=-2\l \Delta \omega, y\de_{xy}\omega\r,
\end{equation}
while for the second term we compute
\begin{equation}
\l y^3\de_{xx} \omega ,\de_y\omega\r+\l y\de_x \omega,\de_y(y^2\de_x \omega )\r=2\| y\de_x \omega\|^2.
\end{equation}
Lastly, the third term yields
\begin{equation}
 \l y\de_{xx}\psi,\de_y\Delta\psi\r+\l y\de_x \Delta\psi,\de_{xy}\psi\r=-\l \Delta(y\de_{x}\psi),\de_{xy}\psi\r+\l y\de_x \Delta\psi,\de_{xy}\psi\r=-2\| \de_{xy}\psi\|^2,
\end{equation}
and \eqref{eq:crossterm} follows. For \eqref{eq:gamma-weight}, 
 we have
\begin{align}
\frac12\ddt\| y\de_x \omega\|^2
&=\l y\de_x \omega, y\de_x (\nu\Delta \omega-y^2\de_x \omega +2\de_x\psi)\r\notag\\
&=\nu\l y\de_x \omega, y \de_{x} \Delta \omega\r + 2\ip{y\de_x\Delta\psi,y\de_{xx}\psi}\notag\\
&=\nu\norm{\de_x\omega}^2-\nu\norm{y\de_x\nabla\omega}^2-4\ip{y\de_{xy}\psi,\de_{xx}\psi}\label{eq:boh},
 \end{align}
 while for \eqref{eq:gamma-psi} we multiply \eqref{eq:poison} by $\de_{xx}\psi$ and obtain
\begin{align}
\l\de_t\omega,\de_{xx}\psi\r +\l y^2\de_x \omega,\de_{xx}\psi\r -2\l\de_x\psi,\de_{xx}\psi\r=\nu\l\Delta\omega,\de_{xx}\psi\r.
\end{align}
Hence
\begin{align}
\frac12\ddt\| \nabla \de_{x}\psi\|^2 +\l y^2\de_x \Delta \psi,\de_{xx}\psi\r =-\nu \norm{\de_x\omega}^2,
\end{align}
and the conclusion follows from
\begin{align}
 \l y^2\de_x \Delta \psi,\de_{xx}\psi\r=-2\l y\de_{xy} \psi,\de_{xx}\psi\r-\l y^2\de_x \nabla \psi,\nabla\de_{xx}\psi\r=-2\l y \de_{xy} \psi,\de_{xx}\psi\r.
\end{align}
The proof is over.
\end{proof}

\subsection{The hypocoercivity setting}\label{ssec:hypoco}
In order to simplify notation, 
we will consider the solution $\omega$ to \eqref{eq:poison} as concentrated on a single $x$-frequency band $\pm k$ ($k\in\N$). We will therefore omit the subscript $k$ in all the quantities, and we will only keep the dependence on $k$ of the various constants. The first step is to define the functional in \eqref{eq:def_Phi_k}.
For $\alpha,\beta, \gamma>0$ to be determined, we let 
\begin{align}
\Phi=\frac12\left[\|\omega\|^2 + \alpha \|\nabla\omega\|^2+4\beta \l y\de_x \omega,\de_y\omega\r +\gamma\| y\de_x \omega\|^2+
2\gamma \| \nabla \de_x\psi\|^2 \right].
\end{align}
Notice that, up to rescaling of the various coefficients, $\Phi$ has exactly the form \eqref{eq:def_Phi_k}, as long as we assume that 
$\omega$ is concentrated in one single frequency band $\pm k$.
We will now derive various properties of $\Phi$ and prove \eqref{eq:diffineq_final}.

\begin{proof}[Proof of Theorem \ref{thm:phi_k}]
The proof proceeds through a series of constraints on $\alpha,\beta,\gamma$ that will be stated and verified at the end of the proof. We first observe that we can guarantee that $\Phi\geq 0$. Clearly,
\begin{align}\label{eq:funct}
4\beta \l y\de_x \omega,\de_y\omega\r\leq 4\beta \| y\de_x \omega\|\|\de_y\omega\| \leq \frac{\alpha}{2}\|\nabla\omega\|^2
+ \frac{8\beta^2}{\alpha}\| y\de_x \omega\|^2
\end{align}
so that if we assume that
\begin{align}\label{eq:constraint}
\frac{\beta^2}{\alpha\gamma}\leq \frac{1}{16}
\end{align}
we obtain the upper and lower bounds
\begin{equation}\label{eq:PhiPOS}
\begin{aligned}
&\frac14\left[2\|\omega\|^2 + \alpha \|\nabla\omega\|^2 +\gamma\| y\de_x \omega\|^2+
4\gamma \| \nabla \de_x\psi\|^2 \right]\leq \Phi,\\
&\Phi \leq \frac14\left[2\|\omega\|^2 + 3\alpha \|\nabla\omega\|^2+3\gamma\| y\de_x \omega\|^2+
4\gamma \| \nabla \de_x\psi\|^2 \right].
\end{aligned}
\end{equation}
By virtue of the energy estimates from Lemma \ref{lem:en_est}, $\Phi$ satisfies
\begin{equation}\label{eq:phi'}
\begin{aligned}
\ddt \Phi &+\nu\|\nabla\omega\|^2+\alpha\nu\|\Delta\omega\|^2+4\beta\| y\de_x \omega\|^2+8\beta\| \de_{xy}\psi\|^2+\gamma\nu\|y \de_{x}\nabla\omega\|^2 +\gamma\nu\|\de_x\omega\|^2\\
&=-4\beta\nu\l \Delta \omega, y\de_{xy}\omega\r-2\alpha\l y\de_x \omega,\de_y\omega\r.
\end{aligned}
\end{equation}
Now, to absorb the first term on the right hand side we make use of \eqref{eq:constraint} and note that
\begin{equation}
\begin{aligned}
-4\beta\nu\l \Delta \omega, y\de_{xy}\omega\r &\leq 4\beta\nu\|\Delta \omega \|\|y\de_{xy}\omega\| \leq \frac{\alpha \nu}{2}\|\Delta \omega \|^2  +2\frac{\beta^2\nu}{\alpha}\|y\de_{xy}\omega\|^2\\
&  \leq \frac{\alpha \nu}{2}\|\Delta \omega \|^2+\frac{\gamma \nu}{8}\|y\de_{x}\nabla\omega\|^2,
\end{aligned}
\end{equation}
while the second term can be estimated as
\begin{align}
-2\alpha\l y\de_x \omega,\de_y\omega\r\leq \frac{\nu}{4} \|\nabla \omega\|^2+ 4\frac{\alpha^2}{\nu} \|y\de_x \omega\|^2.
\end{align}
In view of our aim to reconstruct the functional $\Phi$ on the left hand side of \eqref{eq:phi'}, we observe that the missing term $\norm{\de_{xx}\psi}^2$ may be bounded from above as follows: since
\begin{equation}
 \l \de_{yx}\psi, y\de_x\omega \r=\l \de_{yx}\psi, y\de_{yyx}\psi \r+\l \de_{yx}\psi, y\de_{xxx}\psi \r=-\frac12\|\de_{yx}\psi\|^2+\frac12\|\de_{xx}\psi\|^2
\end{equation}
it follows that
\begin{align}
\|\de_{xx}\psi\|^2-\|\de_{xy}\psi\|^2=2\l\de_{xy}\psi, y\de_{x}\omega \r\leq \|\de_{xy}\psi\|^2+\|y\de_{x}\omega\|^2,
\end{align}
implying
\begin{align}
\|\de_{xx}\psi\|^2\leq\|y\de_x\omega\|^2+2\|\de_{xy}\psi\|^2.
\end{align}
Altogether, this means that we can reduce \eqref{eq:phi'} to
\begin{equation}\label{eq:diffineq1}
\begin{aligned}
\ddt \Phi &+\frac{3\nu}{4}\|\nabla\omega\|^2+\frac{\alpha\nu}{2}\|\Delta\omega\|^2+\left(2\beta-\frac{4\alpha^2}{\nu}\right)\| y\de_x \omega\|^2\\
&+2\beta\| \nabla\de_{x}\psi\|^2+\frac{7}{8}\gamma\nu\|y\de_{x}\nabla\omega\|^2+ \gamma\nu\|\de_x \omega \|^2 \leq  0.
\end{aligned}
\end{equation}
We now specialize the discussion to $\omega$ that is localized on a fixed $x$-frequency band $\pm k$, for some $k\in\N$. Then $\norm{\partial_x\omega}=\abs{k}\norm{\omega}$. We will next choose parameters $\alpha,\beta,\gamma$ in dependence of $\nu$ and $k$ in such a way that a suitable differential inequality will hold. The details are as follows.

We fix the scales of the parameters $\alpha,\beta,\gamma$ with respect to $\nu$ and $k$ to be
\begin{align}\label{eq:coeffchoice}
\alpha=\frac{\nu^{1/2}}{\abs{k}^{1/2}}\alpha_0,\qquad \beta=\frac{1}{\abs{k}}\beta_0,\qquad \gamma=\frac{1}{\nu^{1/2}\abs{k}^{3/2}}\gamma_0,
\end{align}
with $\alpha_0,\beta_0,\gamma_0>0$ independent of $\nu$ and $k$ such that
\begin{align}\label{eq:constraint1}
\frac{\beta_0^2}{\alpha_0\gamma_0}\leq \frac{1}{16},
\end{align}
so that \eqref{eq:constraint1} is satisfied.
In order to reconstruct the $L^2$ norm of $\omega$ on the left hand side of \eqref{eq:diffineq1}, we preliminarily note that an integration by parts yields
\begin{equation}\label{eq:ident}
 \l y \omega, \de_y\omega\r= -\frac12  \|\omega\|^2.
\end{equation}
As a consequence, for every $\sigma>0$ it holds that
\begin{equation}\label{eq:spec-gap}  
\norm{\omega}^2\leq 2 \norm{\de_y\omega}\norm{y \omega}\leq \sigma \norm{\de_y\omega}^2 +\frac{1}{\sigma}\norm{y \omega}^2.
\end{equation}
Taking $\sigma=\frac{1}{2\beta_0^{1/2}}\frac{\nu^{1/2}}{\abs{k}^{1/2}}$, we conclude that
\begin{equation}
 \frac{\beta_0^{1/2}}{2}\nu^{1/2}\abs{k}^{1/2}\norm{\omega}^2\leq \frac{\nu}{4}\norm{\de_y\omega}^2+\beta\norm{y\de_x\omega}^2.
\end{equation}
Upon substituting the relations \eqref{eq:coeffchoice} and assuming the further constraint
\begin{align}\label{eq:constraint2}
\beta_0\geq 4\alpha_0^2,
\end{align}
we thus obtain from \eqref{eq:diffineq1} the inequality
\begin{equation}\label{eq:diffineq4}
 \begin{aligned}
  \ddt \Phi &+\nu^{1/2}\abs{k}^{1/2}\Bigg[ \frac{\beta_0^{1/2}}{2}\norm{\omega}^2
   +\frac{1}{2}\frac{\nu^{1/2}}{\abs{k}^{1/2}}\|\nabla\omega\|^2+\frac{\beta_0-4\alpha_0^2}{\nu^{1/2}\abs{k}^{3/2}}\left(\| y\de_x \omega\|^2 +2\| \nabla\de_{x}\psi\|^2\right)\Bigg] \\
   &+\frac{\alpha\nu}{2}\|\Delta\omega\|^2+\frac{7}{8}\gamma\nu \|y \de_{x}\nabla\omega\|^2+\gamma\nu\|\de_x \omega \|^2\leq 0.
 \end{aligned}
\end{equation}
Now it remains to choose $\alpha_0,\beta_0,\gamma_0>0$ satisfying the constraints \eqref{eq:constraint1} and \eqref{eq:constraint2} and such that the above term in square brackets bounds a multiple of $\Phi$. Factoring out $\beta_0^{1/2}/4$, 
 \begin{align}\label{eq:diffineq4bis}
  \ddt \Phi &+ \frac{\beta_0^{1/2}}{4}\nu^{1/2}\abs{k}^{1/2}\left[2\norm{\omega}^2
   +\frac{2}{\beta_0^{1/2}}\frac{\nu^{1/2}}{\abs{k}^{1/2}}\|\nabla\omega\|^2+\frac{\beta_0-4\alpha_0^2}{\beta_0^{1/2}\nu^{1/2}\abs{k}^{3/2}}\left(3\| y\de_x \omega\|^2 +4\| \nabla\de_{x}\psi\|^2\right)\right]\notag \\
   &+\frac{\alpha\nu}{2}\|\Delta\omega\|^2+\frac{7}{8}\gamma\nu \|y \de_{x}\nabla\omega\|^2+\gamma\nu\|\de_x \omega \|^2\leq 0,
 \end{align}
and invoking \eqref{eq:PhiPOS}, the additional conditions for this read
\begin{equation}\label{eq:constraint3}
 \frac{2}{\beta_0^{1/2}}\geq 3\alpha_0,\qquad \frac{\beta_0-4\alpha_0^2}{\beta_0^{1/2}}\geq \gamma_0.
\end{equation}
It is not hard to check that the choice
\begin{align}
\delta_0^4=\frac{1}{512}, \qquad \alpha_0=4\delta_0^3, \qquad \beta_0=\frac{\delta_0^2}{4}, \qquad \gamma_0=\frac{\delta_0}{4},
\end{align}
satisfies  \eqref{eq:constraint1} with an equality and  \eqref{eq:constraint3}, and hence also \eqref{eq:constraint2} automatically.
With \eqref{eq:diffineq4bis}  this yields
\begin{equation}\label{eq:diffineq5}
  \ddt \Phi +2\eps_0\nu^{1/2}\abs{k}^{1/2}\Phi+\frac{\alpha\nu}{2}\|\Delta\omega\|^2+\frac{7}{8}\gamma \nu\|y \de_{x}\nabla\omega\|^2+\gamma\nu\|\de_x \omega \|^2\leq 0,
\end{equation}
where $\eps_0=\delta_0/4\geq\frac{1}{20}$. This concludes the proof of Theorem \ref{thm:phi_k}.
\end{proof}

\subsection{Semigroup estimates and the proof of Theorem \ref{thm:linenhanced}}\label{sub:semi}
As mentioned above, Theorem \ref{thm:linenhanced} is a direct  consequence of Corollary \ref{cor:fast_decay}. We therefore prove the latter
here below.
\begin{proof}[Proof of Corollary \ref{cor:fast_decay}]
 We use the notation from Theorem \ref{thm:phi_k}. We begin by recalling that per \eqref{eq:PhiPOS}, for $\omega$ localized to $x$-frequency $\pm k$, the quantity
 \begin{equation}
  \frac{1}{2}\norm{\omega}^2 + \frac{\alpha_0\nu^{1/2}}{4\abs{k}^{1/2}}\norm{\nabla\omega}^2 +\frac{\gamma_0\abs{k}^{1/2}}{4\nu^{1/2}}\left[\norm{y \omega}^2 + 4\norm{\nabla \psi}^2\right]
 \end{equation}
 is comparable to $\Phi$. In particular, since $\abs{k}\geq 1$ and also 
 $0<\nu<1$, we have
 \begin{equation}\label{eq:Q_k_bd}
  Q(t):=\frac{1}{2}\norm{\omega(t)}^2 +\frac{\gamma_0}{4}\left[\norm{y \omega(t)}^2 + 2\norm{\nabla \psi(t)}^2\right]\leq \Phi(t).
 \end{equation}
  On the one hand we note that by monotonicity of $\norm{\omega}^2$ and of $\norm{y\de_x\omega}^2+2\norm{\nabla\de_x\psi}^2$ (see \eqref{eq:L2balance} and \eqref{eq:mon} in Remark \ref{rem:mon}) we get
  \begin{equation}
   Q(t)\leq Q(0).
  \end{equation}
  In particular, this suffices to show the claim for $t<T_{\nu,k}:=\frac{1+\abs{\log\nu}+\log\abs{k}}{2\eps_0\nu^{1/2}\abs{k}^{1/2}}$.
  On the other hand, for $t\geq T_{\nu,k}$ we argue as follows: from the energy equality \eqref{eq:L2balance} and the mean value theorem we deduce that there exists 
  \begin{equation}
   t^*\in\left(0,\frac{1}{2\eps_0\nu^{1/2}\abs{k}^{1/2}}\right)
  \end{equation}
  such that
  \begin{equation}
   \frac{\nu^{1/2}}{\abs{k}^{1/2}}\norm{\nabla\omega(t^*)}^2\leq \eps_0\norm{\omega^{in}}^2.
  \end{equation}
  By \eqref{eq:PhiPOS} this implies
  \begin{equation}
   \Phi(t^*)\leq \frac{1}{2}\norm{\omega(t^*)}^2 + 3\alpha_0\eps_0\norm{\omega^{in}}^2 +\frac{\gamma_0\abs{k}^{1/2}}{4\nu^{1/2}}\left[3\norm{y \omega(t^*)}^2 + 4\norm{\nabla \psi(t^*)}^2\right].
  \end{equation}
  Invoking again the aforementioned monotonicity yields
  \begin{equation}
   \Phi(t^*)\leq (3+6\alpha_0\eps_0) \frac{\abs{k}^{1/2}}{\nu^{1/2}}Q(0).
  \end{equation}
  From the differential inequality \eqref{eq:diffineq_final} for $\Phi$ and the fact that $0<t^*<T_{\nu,k}$ it then follows that for $t\geq T_{\nu,k}$ we have
  \begin{equation}
  \begin{aligned}
   Q(t)&\leq \Phi(t)\leq \e^{-2\eps_0\nu^{1/2}\abs{k}^{1/2}(t-t^*)}\Phi(t^*)\leq (3+6\alpha_0\eps_0)\e^{2\eps_0\nu^{1/2}\abs{k}^{1/2}t^*} \frac{\abs{k}^{1/2}}{\nu^{1/2}}\e^{-2\eps_0\nu^{1/2}\abs{k}^{1/2}t}Q(0)\\
   &\leq \e (3+6\alpha_0\eps_0)  \e^{-2\eps_0\frac{\nu^{1/2}\abs{k}^{1/2}}{1+\abs{\log\nu}+\log\abs{k}}t} Q(0),
  \end{aligned}
  \end{equation}
  where we used in the last inequality that $0<t^*<\frac{1}{2\eps_0\nu^{1/2}\abs{k}^{1/2}}$ and that
  \begin{align}
   \frac{\abs{k}^{1/2}}{\nu^{1/2}}\e^{-2\eps_0\nu^{1/2}\abs{k}^{1/2}t}\leq \e^{-2\eps_0\frac{\nu^{1/2}\abs{k}^{1/2}}{1+\abs{\log\nu}+\log\abs{k}}t},\quad t\geq T_{\nu,k}.
  \end{align}
This concludes the proof.
\end{proof}

\section{Nonlinear transition stability threshold}\label{sec:nonlin}
Let us now return to the study of the full, nonlinear equation \eqref{eq:NSEpois}, rewritten here for convenience:
\begin{equation}\label{eq:fullPoison}
 \de_t\omega=\Lin\omega-u\cdot\nabla\omega,\quad u=\nabla^\perp \Delta^{-1}\omega.
\end{equation}
In this section we will see that there is a threshold for the Reynolds number, above which the \emph{nonlinear} flow globally exhibits the enhanced dissipation demonstrated for the linear flow in Section \ref{sec:LinEn}. Here the modes $k=0$ and $k\neq 0$ play different roles, so to give the precise result we introduce the notation $\PP$ for the projection onto the $x$-frequency $k=0$, i.e.
\begin{equation}
 \PP f (x,y)=\PP f(y)=\int_{\T} f(x,y)\dd x.
\end{equation}
We will furthermore denote the shear part of a function by $f_s(y):=\PP f(y)$ and $\tilde{f}(x,y):=\NP f(x,y)$.\footnote{Note also that since $x\in\T$ we have $\norm{f_s}_{L^2}=\sqrt{2\pi}\norm{f_s}_{L^2_y}$.} When we apply $\PP$ to the velocity $u$, then $u_s$ will only denote the scalar $\PP u^1$, since by periodicity
for the second component we have $\PP u^2=0$.
We also remind the reader of the notation \eqref{eq:normX} for the norm $X$: $\norm{f}_X^2=\norm{f}_{L^2}^2+\norm{yf}_{L^2}^2$. 

The main goal of the section is to prove the following theorem, which is nothing but a restatement of Theorem \ref{thm:threshold-main}. We remind the reader that $C_0,\eps_0$ are the constants of the semigroup estimate in Corollary \ref{cor:linenhanced}.
\begin{theorem}\label{thm:threshold}
 For every $\mu,C_1>0$, there exists $\nu_0>0$ such that if $0<\nu\leq\nu_0$ and 
 \begin{equation}
  \norm{\omega^{in}}_{X} + \norm{y \PP u^{1,in}}_{L^2}\leq C_1\nu^{3/4+2\mu},
 \end{equation}
 then there exists a unique global solution $\omega$ to the initial value problem for \eqref{eq:fullPoison} with $\omega(0)=\omega^{in}$. 
 Moreover, the modes with $k\neq 0$ exhibit an enhanced dissipation rate as in the linear case, namely
 \begin{equation}
  \norm{\NP\omega(t)}_{X}\leq 2C_0\e^{-\frac{\eps_0}{1+\log2C_0}\lambda_\nu t}\norm{\NP\omega^{in}}_{X},\qquad \forall t\geq 0,
 \end{equation}
 where $\NP=1-\PP$, while the $k=0$ mode remains bounded, i.e.\ for a universal constant $C_2>0$ we have
 \begin{equation}
  \norm{\PP\omega(t)}_{X}\leq C_2\norm{\omega^{in}}_{X},\qquad \forall t\geq 0.
 \end{equation}
\end{theorem}

The proof proceeds via a bootstrap argument, and relies on the decay properties of the semigroup $\e^{t\Lin}$. This is what motivates our use of the $X$ norm. We thus begin in Section \ref{ssec:basics} by discussing its fundamental properties as relevant to our setting: the enhanced dissipation estimate and the local-wellposedness of \eqref{eq:fullPoison} in this norm. In order to globally propagate control of the $X$ norm, a finer understanding of the structure of the equations is crucial. This is achieved in Section \ref{ssec:struct}. It is followed in Section \ref{ssec:prelim} by a discussion of preliminaries for the proof of the theorem. We show how bounds for the $X$ norm will give us control over the evolution of some quantities that will be essential later on. Finally, Section \ref{ssec:thresh-proof} gives the details for the bootstrap argument that proves Theorem \ref{thm:threshold}.

\subsection*{Notation conventions for Section \ref{sec:nonlin}}

Unlike the case of Section \ref{sec:LinEn}, here we will not be working exclusively with $L^2$ norms, so we shall always specify them. The parameters $\eps_0$ and $C_0$ being initially fixed in size, we will generically denote by $C$ a positive constant that may depend on them. Note that the value of $C$ may change from line to line. In contrast, for further small parameters $0<\delta,\theta\ll 1$ to be chosen later, we shall track their influence by denoting by $C_a$ a constant of the form $C_a=a^{-1}C$, with $a\in\{\delta,\theta\}$.

\subsection{Basic setup}\label{ssec:basics}
In view of the properties of the linear equation $\de_t f=\Lin f$ we introduce the following definition.
\begin{definition}
 For fixed $\eps_0\geq \frac{1}{20}$ and $\lambda_\nu=\frac{\nu^{1/2}}{1+\abs{\log\nu}}$ as in Corollary \ref{cor:linenhanced}, we write
 \begin{equation}
 \norm{f}_{X_t}:=\e^{\eps_0 \lambda_\nu t}\norm{f}_X,\quad \norm{f}_{X[0,T]}:=\sup_{t\in[0,T]}\norm{f}_{X_t},\quad t,T>0.
\end{equation}
\end{definition}
This way the enhanced dissipation of Corollary \ref{cor:linenhanced} can be restated as the following bound for the semigroup $\e^{t\Lin}$:
 \begin{equation}
  \|\e^{t\Lin}\tilde{f}\|_{X_t}\leq C_0\|\tilde{f}\|_{X}.
 \end{equation}
Moreover, the equation \eqref{eq:fullPoison} is locally well-posed in this functional framework.
\begin{lemma}\label{lem:lwp}
 The Poiseuille equation \eqref{eq:fullPoison} is locally (in time) well-posed in $C_t X$. In particular, the mapping $t\mapsto\norm{\tilde\omega(t)}_{X}$ is continuous for a suitably small range of $0\leq t\leq T$ (depending only on $\norm{\omega^{in}}_{X}$ and $\nu$).
\end{lemma}

We remark that clearly this also implies the continuity of $t\mapsto\norm{\tilde\omega(t)}_{X_t}$ wherever it is defined. 
\begin{proof}
 We content ourselves with giving the relevant a priori estimates for $\norm{\omega}_{L^2}$ and $\norm{y\omega}_{L^2}$, which can be rigorously
 justified within a proper approximation scheme.
 For $\omega$ these are just the energy estimates
 \begin{equation*}
  \frac12\ddt \norm{\omega}_{L^2}^2 +\nu\norm{\nabla \omega}_{L^2}^2=0,
 \end{equation*}
 whereas for $y\omega$ we take the inner product of $\eqref{eq:fullPoison}$ with $y^2\omega$ and compute that (compare also \eqref{eq:boh})
 \begin{align}
\frac12\ddt\| y \omega\|^2
&=\l y \omega, y (\nu\Delta \omega-y^2\de_x \omega +2\de_x\psi -u\cdot\nabla \omega)\r\notag\\
&=\nu\norm{\omega}^2-\nu\norm{y\nabla\omega}^2+2\ip{y\omega,y\de_x\Delta^{-1}\omega}- \l y^2\omega ,u\cdot\nabla \omega\r.
 \end{align}
  First of all, notice that
\begin{align}
\Delta (y\Delta^{-1}\omega )=2\Delta^{-1}\de_y\omega+y\omega.
\end{align}
Thus
\begin{align}
y\Delta^{-1}\omega =2\Delta^{-2}\de_y\omega+\Delta^{-1}(y\omega),
\end{align}
implying the commutator relation
\begin{align}\label{eq:comm1}
[y,\Delta^{-1}]\omega:=y\Delta^{-1}\omega-\Delta^{-1}(y\omega)=2\Delta^{-2}\de_y\omega.
\end{align}
Now
\begin{align}
2\ip{y\omega,y\de_x\Delta^{-1}\omega}&=2\ip{y\tilde\omega,y\de_x\Delta^{-1}\tilde\omega}=
2\ip{y\tilde\omega,[y,\Delta^{-1}]\de_x\tilde\omega}+2\ip{y\tilde\omega,\de_x\Delta^{-1}y\tilde\omega}\\
&=4\ip{y\tilde\omega,\Delta^{-2}\de_y\de_x\tilde\omega}=-4\ip{\Delta^{-1}\de_y(y\tilde\omega),\de_x\Delta^{-1}\tilde\omega},
\end{align}
and thus
\begin{align}
2|\ip{y\omega,y\de_x\Delta^{-1}\omega}|=4|\ip{\Delta^{-1}\de_y(y\tilde\omega),\Delta^{-1}\de_x\tilde\omega}|
\leq 4 \| y\omega\|_{L^2}\|\omega\|_{L^2}.
\end{align}
We now proceed to bound $u$ in $L^\infty$. For this we need the one-dimensional Gagliardo-Nirenberg-Sobolev inequality
 \begin{equation}\label{lem:interpol1}
  \norm{u_s}_{L^\infty_y}\leq c_1 \norm{u_s}_{L^2_y}^{1/2}\norm{u_s}_{\dot{H}^1_y}^{1/2},
 \end{equation}
for some $c_1>0$,
and the following standard interpolation result
 \begin{equation}\label{lem:interpol2}
 \|\tilde{u}\|_{L^\infty}\leq\frac{c_2}{\delta}\|\tilde{u}\|_{L^2}^{1-\delta}\|\nabla\tilde{\omega}\|_{L^2}^\delta\leq\frac{c_2}{\delta}\|\tilde{\omega}\|_{L^2}^{1-\delta}\|\nabla\tilde{\omega}\|_{L^2}^\delta,
 \end{equation}
 holding for some $c_2>0$ and all $0<\delta \ll 1$. Arguing as in \eqref{eq:ident}, we easily deduce that 
 \begin{align}
 \|u_s\|_{L^2_y}\leq 2\| y\de_yu_s\|_{L^2_y}=2\| y\omega_s\|_{L^2_y}, 
 \end{align}
 and therefore
  \begin{equation}
  \norm{u}_{L^\infty}\leq C \norm{\omega}_{L^2}^{1/2}\left[\norm{y\omega}_{L^2}^{1/2}+\norm{\nabla\omega}_{L^2}^{1/2}\right].
  \end{equation}
We can conclude that (below $u^2$ stands for the second component of the vector $u$)
\begin{align}
\abs{\ip{u\cdot\nabla\omega,y^2\omega}}=\abs{\ip{u^2\omega, y\omega}}
&\leq C\left[\norm{y\omega}_{L^2}^{1/2}+\norm{\nabla\omega}_{L^2}^{1/2}\right]\norm{\omega}_{L^2}^{3/2}\norm{y\omega}_{L^2}\\
&\leq \frac{\nu}{2}\norm{\nabla\omega}_{L^2}^2+  \frac{C}{\nu^{1/3}}\norm{\omega}_{L^2}^{2}\norm{y\omega}^{4/3}_{L^2}+C\norm{\omega}_{L^2}^{3/2}\norm{y\omega}_{L^2}^{3/2}.
\end{align}
Thus
 \begin{equation}
 \begin{aligned}
  \frac12\ddt \norm{y\omega}_{L^2}^2+\nu\norm{y\nabla \omega}_{L^2}^2\leq  \frac{\nu}{2}\norm{\nabla\omega}_{L^2}^2 +C\Bigg[&\nu\norm{\omega}_{L^2}^2+\norm{y\omega}_{L^2}\norm{\omega}_{L^2}\\
  & +\frac{1}{\nu^{1/3}}\norm{\omega}_{L^2}^{2}\norm{y\omega}^{4/3}_{L^2}+\norm{\omega}_{L^2}^{3/2}\norm{y\omega}_{L^2}^{3/2}\Bigg],
 \end{aligned}
 \end{equation}
 and therefore
 \begin{equation}
  \ddt \left[\norm{\omega}_{L^2}^2+\norm{y\omega}_{L^2}^2\right]\leq C\norm{\omega}_{L^2}^2+\norm{y\omega}_{L^2}\norm{\omega}_{L^2}+  \frac{C}{\nu^{1/3}}\norm{\omega}_{L^2}^{2}\norm{y\omega}^{4/3}_{L^2}+C\norm{\omega}_{L^2}^{3/2}\norm{y\omega}_{L^2}^{3/2}.
 \end{equation}
This provides the key local-in-time a priori estimate on $\|\omega\|_X$, and allows us to conclude the proof.
\end{proof}

\subsection{Structure of the equations and energy estimates}\label{ssec:struct}
Equation \eqref{eq:fullPoison} satisfies the following energy estimates for $\omega$ and $u$:
\begin{equation}\label{eq:full_energy}
 \begin{aligned}
  &\frac12\ddt \norm{u}_{L^2}^2 +\nu\norm{\nabla u}_{L^2}^2=0,\\
  &\frac12\ddt \norm{\omega}_{L^2}^2 +\nu\norm{\nabla \omega}_{L^2}^2=0.
 \end{aligned}
\end{equation}
To go further we will investigate the structure of the equations for the modes $k=0$ and $k\neq 0$.

\subsubsection{$k=0$ mode}
One computes directly from the structure of the Biot-Savart law that $u_s\de_x\omega_s=0$, so that we have $\PP(u\cdot\nabla\omega)=\PP(\tilde u\cdot\nabla\tilde\omega)+\PP(\tilde u^2\de_y\omega_s)+\PP(u_s\de_x\tilde\omega)=\PP(\tilde u\cdot\nabla\tilde\omega)$, i.e.\ there are no self-interactions of the $k=0$ mode. For $\omega_s$ we thus simply have the equation
\begin{equation}\label{eq:omega_s}
 \de_t\omega_s+\PP(\tilde{u}\cdot\nabla\tilde{\omega})=\nu\de_{yy}\omega_s.
\end{equation}
In addition, we observe the following structure of the nonlinearity in this case:
\begin{equation}
\begin{aligned}
 \PP(\tilde{u}\cdot\nabla\tilde{\omega})&=\PP(\nabla\cdot(\tilde{u}\tilde{\omega}))=\int_\T\nabla\cdot(\tilde{u}\tilde{\omega})\dd x=\de_y\int_\T\tilde{u}^2\tilde{\omega}\dd x=\de_y\PP(\tilde{u}^2\tilde{\omega})\\
 &=\de_y\int_\T(\de_x\tilde{\psi}\de_{yy}\tilde{\psi})\dd x=\de_{yy}\int_\T(\de_x\tilde{\psi}\de_{y}\tilde{\psi})\dd x= -\de_{yy}\PP(\tilde{u}^1\tilde{u}^2),
\end{aligned} 
\end{equation}
where we used that $\tilde{\omega}=\de_{xx}\tilde{\psi}+\de_{yy}\tilde{\psi}$. The equations for $u_s$ and $\psi_s$ thus read
\begin{align}
  \de_t u_s+\PP(\tilde{u}^2\tilde{\omega})&=\nu\de_{yy}u_s,\\
  \de_t \psi_s-\PP(\tilde{u}^1\tilde{u}^2)&=\nu\de_{yy}\psi_s.
 \end{align}
With this we easily obtain additional energy estimates: For weights in $y$ this yields
\begin{equation}\label{eq:yomega_s}
 \frac{1}{2}\ddt\norm{y\omega_s}_{L^2}^2=\nu\norm{\omega_s}_{L^2}^2-\nu\norm{y\de_y \omega_s}_{L^2}^2-\ip{\PP(\tilde{u}\cdot\nabla\tilde{\omega}),y^2\omega_s}
\end{equation}
and
\begin{equation}\label{eq:yu_s}
 \frac{1}{2}\ddt\norm{yu_s}_{L^2}^2=\nu\norm{u_s}_{L^2}^2-\nu\norm{y\de_y u_s}_{L^2}^2-\ip{\PP(\tilde{u}^2\tilde{\omega}),y^2u_s}.
\end{equation}
To control the term with the positive sign on the right hand side of \eqref{eq:yomega_s} we will use \eqref{eq:full_energy}, whereas for that in \eqref{eq:yu_s} we compute the $\dot{H}_y^{-1}$ norm of $u_s$ as follows:
\begin{equation}\label{eq:psi_s}
 \frac{1}{2}\ddt\norm{\psi_s}_{L^2}^2=-\nu\norm{u_s}_{L^2}^2+\ip{\PP(\tilde{u}^1\tilde{u}^2),\psi_s}.
\end{equation}

\subsubsection{$k\neq 0$ modes}
A useful consequence in our setup is that
\begin{equation}
 \norm{\tilde{u}}_{L^2}\leq C\norm{\tilde\omega}_{X},
\end{equation}
which can be seen as follows: since $\norm{\tilde{u}}_{L^2}\leq \sqrt{2}\norm{y\tilde{u}}_{L^2}^{1/2}\norm{\de_y\tilde{u}}_{L^2}^{1/2}$ (see also \eqref{eq:spec-gap}) it will suffice to bound $\norm{\de_y\tilde{u}}_{L^2}\leq\norm{\tilde\omega}_{L^2}$, and by commuting $y$ with the Biot-Savart law we get $\norm{y\tilde{u}}_{L^2}\leq3\norm{\tilde\omega}_{L^2}+\norm{y\tilde\omega}_{L^2}$.
Moreover, we have the following energy estimates on $\tilde\omega$.
\begin{lemma}\label{lem:tilde-en}
 There exists $C>0$ such that for $\delta>0$ there holds
 \begin{equation}
  \int_0^t\norm{\nabla\tilde\omega}_{L^2}^2 \leq\frac{C}{\delta} \nu^{-1}(\nu^{-1-\delta}\norm{\omega^{in}}_{L^2}^2)^{\frac{1}{1-\delta}}\int_0^t\norm{\tilde\omega}_{L^2}^2 +\nu^{-1}\norm{\omega^{in}}_{L^2}^2,\quad 0\leq t\leq T.
 \end{equation}
\end{lemma}

\begin{proof}
 We have
 \begin{equation}
  \frac{1}{2}\ddt\norm{\tilde\omega}_{L^2}^2=-\nu\norm{\nabla\tilde\omega}_{L^2}^2-\ip{u\cdot\nabla\omega,\tilde\omega}.
 \end{equation}
 To bound the nonlinearity we notice that by the divergence structure and since $u_s\cdot\nabla\omega_s=0$ there holds
 \begin{equation}
  \ip{u\cdot\nabla\omega,\tilde\omega}=\ip{u\cdot\nabla\omega_s,\tilde\omega}=-\ip{u\omega_s,\nabla\tilde\omega}=-\ip{\tilde{u}\omega_s,\nabla\tilde\omega},
 \end{equation}
 so that for $\delta>0$ we obtain the bound
 \begin{equation}
  \abs{\ip{u\cdot\nabla\omega,\tilde\omega}}\leq\norm{\tilde{u}}_{L^\infty}\norm{\omega_s}_{L^2}\norm{\nabla\tilde\omega}_{L^2}\leq \frac{C}{\delta} (\nu^{-1-\delta}\norm{\omega^{in}}_{L^2}^2)^{\frac{1}{1-\delta}}\norm{\tilde\omega}_{L^2}^2+\frac{\nu}{2}\norm{\nabla\tilde\omega}_{L^2}^2.
 \end{equation}
\end{proof}

\subsubsection{Second order derivatives}
For future use we state the following
\begin{lemma}\label{lem:D2omega}
 Assume $\omega$ solves \eqref{eq:fullPoison}. Then
 \begin{equation}
  \int_0^t\norm{D^2\omega}_{L^2}^2 \leq 16\nu^{-2}\int_0^t\norm{y\tilde\omega}_{L^2}^2 +4C\nu^{-3}\norm{\omega^{in}}_{L^2}^4 + \nu^{-1}\norm{\nabla\omega^{in}}_{L^2}^2.
 \end{equation}
\end{lemma}

\begin{proof}
 We recall that $\norm{D^2\omega}_{L^2}=\norm{\Delta\omega}_{L^2}$. Testing the equation $\de_t\omega+u\cdot\nabla\omega=\Lin\omega$ with $\Delta\omega$ then yields
 \begin{equation}
  \nu\int_0^t\norm{\Delta\omega}_{L^2}^2=-\norm{\nabla\omega(t)}_{L^2}^2+\norm{\nabla\omega^{in}}_{L^2}^2+\int_0^t\ip{y^2\de_x\omega-2\de_x\Delta^{-1}\omega,\Delta\omega}+\ip{u\cdot\nabla\omega,\Delta\omega}.
 \end{equation}
 By antisymmetry and since $\de_x\omega=\de_x\tilde\omega$ we deduce that
 \begin{equation}
  \ip{y^2\de_x\omega-2\de_x\Delta^{-1}\omega,\Delta\omega}=\ip{y^2\de_x\tilde\omega,\de_{yy}\omega}=2\ip{y\tilde\omega,\de_{xy}\omega},
 \end{equation}
 so we can bound
 \begin{equation}
  \int_0^t\abs{\ip{y^2\de_x\omega-2\de_x\Delta^{-1}\omega,\Delta\omega}}\leq 2\int_0^t\norm{y\tilde\omega}_{L^2}\norm{\Delta\omega}_{L^2} \leq\frac{8}{\nu}\int_0^t\norm{y\tilde\omega}_{L^2}^2 +\frac{\nu}{4}\int_0^t\norm{\Delta\omega}_{L^2}^2 .
 \end{equation}
 
 For the nonlinear term we notice that by the divergence structure we have
 \begin{equation}
  \ip{u\cdot\nabla\omega,\Delta\omega}=-\sum_{i,j}\ip{\de_ju^i\de_i\omega,\de_j\omega},
 \end{equation}
 and thus
 \begin{equation}
 \begin{aligned}
  \int_0^t\abs{\ip{u\cdot\nabla\omega,\Delta\omega}}&\leq\norm{\omega}_{L^\infty_t L^2}\int_0^t\norm{\nabla\omega}_{L^4}^2\leq C \norm{\omega}_{L^\infty_t L^2}\int_0^t\norm{\nabla\omega}_{L^2}\norm{\Delta\omega}_{L^2}\\
  &\leq \frac{4C\norm{\omega}_{L^\infty_t L^2}^2}{\nu}\int_0^t\norm{\nabla\omega}_{L^2}^2+\frac{\nu}{4}\int_0^t\norm{\Delta\omega}_{L^2}^2.
 \end{aligned} 
 \end{equation}
 Since $\norm{\omega}_{L^\infty_t L^2}\leq\|\omega^{in}\|_{L^2}$ by \eqref{eq:full_energy}, this yields the claim.
\end{proof}

\subsection{Preliminaries of the proof of Theorem \ref{thm:threshold}}\label{ssec:prelim}
As announced, here we will see how control of $\norm{\tilde{\omega}}_{X_t}$ gives bounds for the $k=0$ mode and second order derivatives, as well as an energy estimate for $\tilde\omega$ (Proposition \ref{prop:prelim}). We emphasize that the estimates here are global in nature, i.e.\ hold as long as one has suitable bounds on $\tilde\omega$ and do not need to be proved by iteration -- this is in contrast to the bootstrap for $\tilde\omega$.

\begin{proposition}\label{prop:prelim}
 Let $\omega$ solve \eqref{eq:fullPoison}, and let $T>0$. Assume that $\norm{\tilde\omega}_{X[0,T]}\leq C\norm{\tilde\omega^{in}}_{X}$ for some $C>0$. Then for $0\leq t\leq T$ we have:
 \begin{enumerate}
  \item Energy estimates for $\tilde\omega$: 
  \begin{equation}
   \int_0^t\norm{\nabla\tilde\omega}_{L^2}^2\leq C_\delta \nu^{-1}(\nu^{-1-\delta}\norm{\omega^{in}}_{L^2}^2)^{\frac{1}{1-\delta}}\lambda_\nu^{-1}\norm{\tilde\omega^{in}}_{X}^2 +\nu^{-1}\norm{\omega^{in}}_{L^2}^2,\quad 0<\delta\ll 1.
  \end{equation}
  \item Control of second order derivatives of $\omega$:
  \begin{equation}\label{eq:2nd_deriv}
  \int_0^t\norm{D^2\omega}_{L^2}^2 \leq C\nu^{-2}\lambda_\nu^{-1}\norm{\tilde\omega^{in}}_{X}^2+4C\nu^{-3}\norm{\omega^{in}}_{L^2}^4+\nu^{-1}\norm{\nabla\omega^{in}}_{L^2}^2.
  \end{equation}
  \item Lower order (and weighted) energy estimates for $u_s$:
  \begin{equation}\label{eq:yu_s-est}
  \norm{yu_s(t)}_{L^2}+\norm{\psi_s(t)}_{L^2}\leq \norm{yu_s^{in}}_{L^2}+\norm{\psi_s^{in}}_{L^2}+C\nu^{-1/4}\lambda_\nu^{-3/4}\norm{\omega^{in}}_{X}^2.
 \end{equation}
  \item Weighted estimates for $y\omega_s$:
  \begin{equation}\label{eq:yomega_s-est}
  \begin{aligned}
  \norm{y\omega_s(t)}_{L^2}+\norm{u(t)}_{L^2}&\leq \norm{y\omega_s^{in}}_{L^2}+\norm{u^{in}}_{L^2}+ +C\lambda_\nu^{-1/2}\norm{\omega^{in}}_{X}\left( \int_0^t\norm{\nabla\tilde\omega}^2_{L^2} \right)^{1/2}\\
  &\quad +C_\theta^2[\lambda_\nu^{-1/2}\nu^{(-1+\theta)/2}]\left(\int_0^t\norm{D^2\tilde\omega}_{L^2}^2\right)^{\theta/2},\quad 0<\theta\ll 1.
 \end{aligned}
 \end{equation}
 \end{enumerate}
\end{proposition}

\begin{remark}[Bounds in the setting of Theorem \ref{thm:threshold}]\label{rem:bounds}
  The assumptions of Theorem \ref{thm:threshold} imply that 
  \begin{equation}
   \norm{u_s^{in}}_{L^2}+\norm{\psi_s^{in}}_{L^2}\leq C\nu^{3/4+2\mu},
  \end{equation} 
  and hence also $\norm{u^{in}}_{L^2}\leq C\nu^{3/4+2\mu}$. Furthermore, as a consequence of the energy inequality \eqref{eq:full_energy} for $\omega$ we may suppose without loss of generality that $\norm{\nabla\omega^{in}}_{L^2}^2\leq C\nu^{2\mu}$.
  Under the additional assumptions of smallness of the additional data as in Theorem \ref{thm:threshold}, this yields the bounds 
  \begin{equation}\label{eq:useful_bds}
  \begin{aligned}
   &\int_0^t\norm{D^2\omega}_{L^2}^2 \leq C\nu^{-1+2\mu},\quad \int_0^t\norm{\nabla\tilde\omega}_{L^2}^2 \leq C \nu^{-1}\norm{\tilde\omega^{in}}_{X}^2,\\
   &\norm{yu_s(t)}_{L^2}+\norm{\psi_s(t)}_{L^2}+\norm{y\omega_s(t)}_{L^2}+\norm{u(t)}_{L^2}\leq C \nu^{3/4+2\mu}.
  \end{aligned}
  \end{equation}
\end{remark}

\begin{proof}[Proof of Proposition \ref{prop:prelim}]
 Items (1) and (2) follow directly by inserting the assumption into Lemmas \ref{lem:tilde-en} and \ref{lem:D2omega}, respectively.
 For (3), we compute from the above energy estimates \eqref{eq:yu_s} and \eqref{eq:psi_s} that
 \begin{equation}
  \frac{1}{2}\ddt [\norm{yu_s}_{L^2}^2+\norm{\psi_s}_{L^2}^2]= -\nu\norm{y\de_y u_s}_{L^2}^2-\ip{\PP(\tilde{u}^2\tilde{\omega}),y^2u_s}-\ip{\PP(\tilde{u}^1\tilde{u}^2),\psi_s}.
 \end{equation}
 This implies the bound 
 \begin{equation}
  \frac{1}{2}\ddt [\norm{yu_s}_{L^2}^2+\norm{\psi_s}_{L^2}^2]\leq -\nu\norm{y\de_y u_s}_{L^2}^2+\norm{\tilde{u}^2}_{L^\infty}\norm{y\tilde{\omega}}_{L^2}\norm{yu_s}_{L^2}+\norm{\tilde{u}^2}_{L^\infty}\norm{\tilde{u}^1}_{L^2}\norm{\psi_s}_{L^2},
 \end{equation}
 and thus
 \begin{equation}
  \norm{yu_s(t)}_{L^2}+\norm{\psi_s(t)}_{L^2}\leq \norm{yu_s^{in}}_{L^2}+\norm{\psi_s^{in}}_{L^2}+\int_0^t \norm{\tilde{u}^2}_{L^\infty}[\norm{y\tilde{\omega}}_{L^2}+\norm{\tilde{u}^1}_{L^2}].
 \end{equation}
 Next we use the interpolation inequality $\norm{\tilde{u}}_{L^\infty}\leq C\norm{\tilde{u}}_{L^2}^{1/2}\norm{\nabla\tilde\omega}_{L^2}^{1/2}$ to conclude that
 \begin{equation}
 \begin{aligned}
  &\int_0^t \norm{\tilde{u}}_{L^\infty}[\norm{y\tilde{\omega}}_{L^2}+\norm{\tilde{u}^1}_{L^2}] \leq C \int_0^t\norm{\nabla\tilde\omega}_{L^2}^{1/2}\norm{\tilde\omega}_{X_\tau}^{3/2}\dd \tau\\
  &\qquad \leq C\left(\int_0^t\norm{\nabla\omega}_{L^2}^2 \right)^{1/4}\left(\int_0^t\norm{\tilde\omega}_{X_\tau}^2 \dd \tau\right)^{3/4} \leq C \nu^{-1/4}\lambda_\nu^{-3/4}\norm{\omega^{in}}_{L^2}^{1/2}\norm{\omega^{in}}_{X}^{3/2},
 \end{aligned}
 \end{equation}
 since $\norm{\tilde{u}}_{L^2}\leq\norm{\tilde\omega}_{X}\leq \e^{-\eps_0\lambda_\nu t}\norm{\tilde\omega}_{X_t}$. This gives the claim.
 
 For (4), we compute from the above energy estimates \eqref{eq:yomega_s} and \eqref{eq:full_energy} that
 \begin{equation}
  \frac{1}{2}\ddt [\norm{y\omega_s}_{L^2}^2+\norm{u}_{L^2}^2]\leq -\nu\norm{y\de_y \omega_s}_{L^2}^2-\ip{\PP(\tilde{u}\cdot\nabla\tilde{\omega}),y^2\omega_s}.
 \end{equation}
 Here we will bound
 \begin{equation}
  \abs{\ip{\PP(\tilde{u}\cdot\nabla\tilde{\omega}),y^2\omega_s}}\leq\norm{y\tilde{u}\cdot\nabla\tilde\omega}_{L^2}\norm{y\omega_s}_{L^2},
 \end{equation}
 so that
 \begin{equation}
  \norm{y\omega_s(t)}_{L^2}+\norm{u(t)}_{L^2}\leq \norm{y\omega_s^{in}}_{L^2}+\norm{u^{in}}_{L^2}+\int_0^t\norm{y\tilde{u}\cdot\nabla\tilde\omega}_{L^2}.
 \end{equation}
 To bound the nonlinearity in $L^2$ we proceed as follows.
We begin by invoking \eqref{eq:comm1} and noting that
\begin{equation}\label{eq:y-nonlin}
 y(u\cdot\nabla\omega)=y(\nabla^\perp\Delta^{-1}\omega\cdot\nabla\omega)=\nabla^\perp\Delta^{-1}(y\omega)\cdot\nabla\omega+2\nabla^\perp\de_y\Delta^{-2}\omega\cdot\nabla\omega+\Delta^{-1}\omega\de_x\omega.
\end{equation}
This is particularly useful in the case of $\tilde{u}$ rather than $u$, since in two-dimensions we have the Agmon inequality $\norm{f}_{L^\infty}\leq C\norm{f}_{L^2}^{1/2}\norm{f}_{\dot{H}^2}^{1/2}$, and thus the commutator terms are easily controlled. By boundedness of the Riesz transform we obtain
\begin{equation}\label{eq:easybdd}
 \norm{\nabla^\perp\de_y\Delta^{-2}\tilde{\omega}}_{L^\infty}+\norm{\Delta^{-1}\tilde{\omega}}_{L^\infty}\leq C\norm{\tilde{\omega}}_{L^2}.
\end{equation}
We can thus bound the corresponding terms as above.
On the other hand, the term $\nabla^\perp\Delta^{-1}(y\tilde{\omega})\cdot\nabla\tilde\omega$ will be treated slightly differently. By standard interpolation in 2d we have \begin{equation}\label{eq:interpol1}
 \norm{\nabla^\perp\Delta^{-1}f}_{L^{2/\theta}}\leq C_\theta\norm{\nabla^\perp\Delta^{-1}f}_{L^2}^{\theta}\norm{\nabla^\perp\Delta^{-1}f}_{\dot{H}^1}^{1-\theta},\qquad 0<\theta\ll 1,
\end{equation}
and thus
\begin{equation}\label{eq:easybdd1}
 \norm{\nabla^\perp\Delta^{-1}(y\tilde{\omega})\cdot\nabla\tilde\omega}_{L^2}\leq C_\theta \norm{\nabla^\perp\Delta^{-1}y\tilde\omega}_{L^{2/\theta}}\norm{\nabla\tilde\omega}_{L^{2/(1-\theta)}}\leq C_\theta\norm{y\tilde\omega}_{L^2}\norm{\nabla\tilde\omega}_{L^{2/(1-\theta)}}.
\end{equation}
Similarly we have
\begin{equation}\label{eq:easybdd2}
 \norm{\nabla\tilde\omega}_{L^{2/(1-\theta)}}\leq C_\theta\norm{\nabla\tilde\omega}_{L^2}^{1-\theta}\norm{D^2\tilde\omega}_{L^2}^\theta.
\end{equation}
It follows that 
 \begin{equation}\label{eq:easybdd3}
  \norm{y\tilde{u}\cdot\nabla\tilde\omega}_{L^2}\leq C_\theta^2\norm{y\tilde\omega}_{L^2}\norm{\nabla\tilde\omega}_{L^2}^{1-\theta}\norm{D^2\tilde\omega}_{L^2}^\theta+3\norm{\tilde\omega}_{L^2}\norm{\nabla\tilde\omega}_{L^2}.
 \end{equation}
 It thus follows that
 \begin{equation}
 \begin{aligned}
  \int_0^t\norm{y\tilde{u}\cdot\nabla\tilde\omega}_{L^2}  &\leq C_\theta^2 \int_0^t\norm{y\tilde\omega}_{L^2}\norm{\nabla\tilde\omega}_{L^2}^{1-\theta}\norm{D^2\tilde\omega}_{L^2}^\theta  + 3\int_0^t\norm{\tilde\omega}_{L^2}\norm{\nabla\tilde\omega}_{L^2} \\
  &\leq C_\theta^2 \left(\int_0^t \norm{y\tilde\omega}_{L^2}^2 \right)^{1/2}\left(\int_0^t \norm{\nabla\tilde\omega}_{L^2}^2 \right)^{(1-\theta)/2}\left(\int_0^t\norm{D^2\tilde\omega}_{L^2}^2 \right)^{\theta/2}\\
  &\quad + 3\left(\int_0^t \norm{\tilde\omega}_{L^2}^2 \right)^{1/2}\left( \int_0^t\norm{\nabla\tilde\omega}^2_{L^2} \right)^{1/2} \\
  &\leq C_\theta^2[\lambda_\nu^{-1/2}\nu^{(-1+\theta)/2}]\norm{\omega^{in}}_X^{2-\theta}\left(\int_0^t\norm{D^2\tilde\omega}_{L^2}^2\right)^{\theta/2}\\
 &\quad +C\lambda_\nu^{-1/2}\norm{\omega^{in}}_{X}\left( \int_0^t\norm{\nabla\tilde\omega}^2_{L^2} \right)^{1/2}.
 \end{aligned}
 \end{equation}
\end{proof}

\subsection{Proof of Theorem \ref{thm:threshold} via Bootstrapping}\label{ssec:thresh-proof}
We will prove the theorem by a continuation argument, the key step of which is the following:
\begin{bootstrap*}\label{prop:bootstrap}
 Let $\kappa_0:=\eps_0(1+\log 2C_0)^{-1}$. If $\nu>0$ is small enough, then for $0<T\leq \kappa_0^{-1}\lambda_\nu^{-1}$, the \emph{bootstrap assumption}
 \begin{equation}\label{eq:btstrap_assum}
  \norm{\tilde\omega}_{X[0,T]}\leq 4C_0\norm{\tilde\omega^{in}}_{X}
 \end{equation}
 implies the stronger bound
 \begin{equation}\label{eq:btstrap_concl}
  \norm{\tilde\omega}_{X[0,T]}\leq 2C_0\norm{\tilde\omega^{in}}_{X}.
 \end{equation}
\end{bootstrap*}

\begin{proof}[Proof of Theorem \ref{thm:threshold}]
From this the proof of the theorem follows: $\kappa_0=\eps_0(1+\log 2C_0)^{-1}$ is chosen in order to ensure that $2C_0\e^{-\eps_0\kappa_0^{-1}}\leq \e^{-1}$, and thus
\begin{equation}
 \norm{\tilde\omega(\kappa_0^{-1}\lambda_\nu^{-1})}_{X}\leq \e^{-1}\norm{\tilde\omega^{in}}_{X}.
\end{equation}
Then thanks to the local well-posedness in Lemma \ref{lem:lwp} we can iterate with time intervals of size $\kappa_0^{-1}\lambda_\nu^{-1}$: indeed, we highlight again that the estimates for the $k=0$ mode and the second order derivatives in Proposition \ref{prop:prelim} will continue to hold whenever the (bootstrap) assumption is valid.
We thereby obtain a global solution. Its decay rate can directly be seen to be
\begin{equation}
 \norm{\tilde\omega(t)}_{X}\leq \e^{-\kappa_0 \lambda_\nu t}\norm{\tilde\omega^{in}}_{X}
\end{equation}
for all $t\geq 0$. The bound on the $k=0$ mode follows from \eqref{eq:useful_bds}, thereby concluding the proof.
\end{proof}
It thus remains to prove the validity of the above implication \eqref{eq:btstrap_assum} $\Rightarrow$ \eqref{eq:btstrap_concl} in the setting
of Theorem \ref{thm:threshold}, which in particular assumes that
\begin{align}\label{eq:assumpindat}
\norm{\omega^{in}}_{X} + \norm{y u_s^{in}}_{L^2}\leq C_1\nu^{3/4+2\mu}.
\end{align}

\begin{proof}[Proof of the Bootstrap Step.]
 By Duhamel's formula we have
\begin{equation}
 \tilde{\omega}(t)=\e^{t\Lin}\tilde{\omega}^{in}-\e^{t\Lin}\int_0^t \e^{-\tau\Lin}\NP (u\cdot\nabla\omega)(\tau)\dd \tau,
\end{equation}
and hence
\begin{equation}
 \norm{\tilde{\omega}(t)}_{X_t}\leq C_0\norm{\omega^{in}}_{X}+C_0\int_0^t \e^{\eps_0 \lambda_\nu \tau}\norm{\NP (u\cdot\nabla\omega)(\tau)}_{X}\dd \tau.
\end{equation}
To conclude the proof it thus suffices to show that for $T_\nu:=\kappa_0^{-1}\lambda_\nu^{-1}$ we have
\begin{equation}
 \abs{\int_0^{T_\nu} \e^{\eps_0 \lambda_\nu \tau}\norm{\NP (u\cdot\nabla\omega)(\tau)}_{X}\dd \tau}\leq \norm{\tilde\omega^{in}}_{X}.
\end{equation}
In fact, we will prove the stronger bound
\begin{equation}\label{eq:stronger}
 \abs{\int_0^{T_\nu} \e^{\eps_0 \lambda_\nu \tau}\norm{\NP (u\cdot\nabla\omega)(\tau)}_{X}\dd \tau}\leq C\nu^{\mu}\norm{\tilde\omega^{in}}_{X},
\end{equation}
where $\mu>0$ and $C>0$ is a constant independent of $\nu$ (but depending on $\eps_0$, $\mu$ and other small, but fixed parameters $0<\delta,\theta\ll 1$). Since we have the freedom to choose $\nu$ small enough this gives the claim.

Towards this end, we note the crude estimate
\begin{align}
\norm{y^a\NP (u\cdot\nabla\omega)}_{L^2}\leq\norm{y^au_s\de_x\tilde\omega}_{L^2}+\norm{y^a\tilde{u}\cdot\nabla\omega}_{L^2},\quad a\in\{0,1\}.
\end{align}
In the remainder of this proof we give the relevant estimates term by term, making frequent use of the bounds stated in \eqref{eq:useful_bds}.

\subsection*{$L^2$ estimates}
Recall from the interpolation inequality \eqref{lem:interpol2} that there exists $C_\delta>0$ such that for $0<\delta\ll 1$ we have
\begin{equation}
 \norm{\tilde{u}}_{L^\infty}\leq C_\delta\norm{\tilde{\omega}}_{L^2}^{1-\delta}\norm{\nabla\tilde{\omega}}_{L^2}^\delta.
\end{equation}
We then use H\"older inequality with $p:=\frac{2}{1-\delta}$,  \eqref{eq:full_energy}, \eqref{eq:useful_bds} and \eqref{eq:btstrap_assum},  to conclude 
\begin{equation}\label{eq:bdutil}
 \begin{aligned}
  &\int_0^{T_\nu} \e^{\eps_0 \lambda_\nu \tau}\norm{\tilde{u}\cdot\nabla\omega}_{L^2}\dd \tau\leq C_\delta\int_0^{T_\nu} \e^{\eps_0 \lambda_\nu \tau}\norm{\tilde\omega}_{L^2}^{1-\delta}\norm{\nabla\tilde\omega}_{L^2}^\delta\norm{\nabla\omega}_{L^2}\dd \tau\\
  &\quad \leq C_\delta\left( \int_0^{T_\nu} \e^{\eps_0 p\lambda_\nu \tau}\norm{\tilde\omega}_{L^2}^2\dd \tau \right)^{1/p}\left( \int_0^{T_\nu} \norm{\nabla\tilde\omega}_{L^2}^2 \right)^{\delta/2}\left( \int_0^{T_\nu} \norm{\nabla\omega}_{L^2}^2 \right)^{1/2}\\
  &\quad \leq C_\delta\left( \int_0^{T_\nu} \e^{\eps_0 (p-2)\lambda_\nu \tau}\norm{\tilde\omega}_{X_\tau}^2\dd \tau \right)^{1/p}\left( \nu^{-1}\norm{\tilde\omega^{in}}_{L^2}^2 \right)^{\delta/2}\left( \nu^{-1}\norm{\omega^{in}}_{L^2}^2 \right)^{1/2}\\
  &\quad \leq C_\delta\nu^{-(1+\delta)/2}\norm{\tilde\omega^{in}}_{X}\norm{\omega^{in}}_{L^2}\left( \int_0^{T_\nu} \e^{\eps_0 (p-2)\lambda_\nu \tau} \dd \tau \right)^{1/p}\\
  &\quad \leq C_\delta^{3/2}\nu^{-(1+\delta)/2}\norm{\tilde\omega^{in}}_{X}\norm{\omega^{in}}_{L^2} \lambda_\nu^{-(1-\delta)/2}\\
  &\quad =C_\delta^{3/2}\nu^{-(3+\delta)/4}\norm{\omega^{in}}_{L^2}\norm{\tilde\omega^{in}}_{X}.
 \end{aligned}
\end{equation}
Hence it suffices to choose $\delta=\mu$ and use \eqref{eq:assumpindat} to obtain the desired estimate
\begin{equation}\label{eq:bdutil2}
 \begin{aligned}
  &\int_0^{T_\nu} \e^{\eps_0 \lambda_\nu \tau}\norm{\tilde{u}\cdot\nabla\omega}_{L^2}\dd \tau\leq C\nu^\mu\norm{\tilde\omega^{in}}_{X}.
 \end{aligned}
\end{equation}
On the other hand, for the second bilinear term we notice that by the interpolation inequality \eqref{lem:interpol1} we have
\begin{equation}
 \norm{u_s}_{L^\infty}=\norm{u_s}_{L^\infty_y}=\norm{\de_y\abs{\de_y}^{-2}\omega_s}_{L^\infty_y}\leq C\norm{\omega_s}_{\dot{H}^{-1}_y}^{1/2}\norm{\omega_s}_{L^2_y}^{1/2},
\end{equation}
so that 
\begin{equation}
\begin{aligned}\label{eq:bdus}
 &\int_0^{T_\nu} \e^{\eps_0 \lambda_\nu \tau}\norm{u_s\de_x\tilde\omega}_{L^2}\dd \tau\leq \int_0^{T_\nu} \e^{\eps_0 \lambda_\nu \tau}\norm{u_s}_{L^\infty}\norm{\nabla\tilde\omega}_{L^2}\dd \tau\\
 &\qquad \leq  C\sup_{0\leq t\leq T_\nu}\left[\norm{\omega_s}_{\dot{H}^{-1}_y}^{1/2}\norm{\omega_s}_{L^2_y}^{1/2}\right]\left(\int_0^{T_\nu} \e^{2\eps_0 \lambda_\nu \tau}\dd \tau \right)^{1/2}\nu^{-1/2}\norm{\tilde\omega^{in}}_{L^2}\\
 &\qquad \leq  C\nu^{-1/2}\lambda_\nu^{-1/2} \norm{u_s}_{L^\infty_t L^2_y}^{1/2}\norm{\omega^{in}}_{L^2}^{1/2}\norm{\tilde\omega^{in}}_{L^2}\\
 &\qquad \leq  C\nu^{-3/4}(1+\abs{\log\nu})^{1/2}\norm{u_s}_{L^\infty_t L^2_y}^{1/2}\norm{\omega^{in}}_{L^2}^{1/2} \norm{\tilde\omega^{in}}_{L^2}.
\end{aligned}
\end{equation}
Appealing to \eqref{eq:useful_bds} and \eqref{eq:assumpindat}, we obtain the desired estimate
\begin{align}\label{eq:bdus2}
 &\int_0^{T_\nu} \e^{\eps_0 \lambda_\nu \tau}\norm{u_s\de_x\tilde\omega}_{L^2}\dd \tau\leq C\nu^\mu\norm{\tilde\omega^{in}}_{X},
\end{align}
completing the first part of the argument.

\subsection*{Weight $y$ in $L^2$} 
As above in \eqref{eq:bdutil} and \eqref{eq:bdus}, we prove separately the bounds involving $\tilde u$ and $u_s$.
We make again use of \eqref{eq:y-nonlin} and \eqref{eq:easybdd}, and arrive at the analogous of \eqref{eq:easybdd3}, namely
 \begin{equation}
  \norm{y\tilde{u}\cdot\nabla\omega}_{L^2}\leq C_\theta^2\norm{y\tilde\omega}_{L^2}\norm{\nabla\omega}_{L^2}^{1-\theta}\norm{D^2\omega}_{L^2}^\theta+3\norm{\tilde\omega}_{L^2}\norm{\nabla\omega}_{L^2}.
 \end{equation}
It follows that
\begin{equation}
 \begin{aligned}
  &\int_0^{T_\nu} \e^{\eps_0 \lambda_\nu \tau}\norm{y\tilde{u}\cdot\nabla\omega}_{L^2}\dd \tau 
\leq C_\theta^2\int_0^{T_\nu} \e^{\eps_0 \lambda_\nu \tau} \left[\norm{y\tilde\omega}_{L^2}\norm{\nabla\omega}_{L^2}^{1-\theta}\norm{D^2\omega}_{L^2}^\theta+\norm{\tilde\omega}_{L^2}\norm{\nabla\omega}_{L^2}\right] \dd \tau.
 \end{aligned}
\end{equation}
Now, the second term can be estimated as in \eqref{eq:bdutil}, obtaining 
\begin{equation}
 \begin{aligned}
\int_0^{T_\nu} \e^{\eps_0 \lambda_\nu \tau}\norm{\tilde\omega}_{L^2}\norm{\nabla\omega}_{L^2} \dd \tau
&\leq C\left( \int_0^{T_\nu} \e^{2\eps_0 \lambda_\nu \tau}\norm{\tilde\omega}_{L^2}^2\dd \tau \right)^{1/2}\left( \int_0^{T_\nu} \norm{\nabla\omega}_{L^2}^2 \right)^{1/2}\\
  & \leq C\left( \int_0^{T_\nu} \norm{\tilde\omega}_{X_\tau}^2\dd \tau \right)^{1/2}\left( \nu^{-1}\norm{\omega^{in}}_{L^2}^2 \right)^{1/2}\\
  & \leq C\nu^{-3/4}(1+\abs{\log\nu})^{1/2}\norm{\tilde\omega^{in}}_{X}\norm{\omega^{in}}_{L^2}.
 \end{aligned}
\end{equation}
while the first term is handled as
\begin{equation}
 \begin{aligned}
&\int_0^{T_\nu} \e^{\eps_0 \lambda_\nu \tau}\norm{y\tilde\omega}_{L^2}\norm{\nabla\omega}_{L^2}^{1-\theta}\norm{D^2\omega}_{L^2}^\theta \dd \tau\\
  &\qquad \leq C_\theta^2 \left(\int_0^{T_\nu} \e^{2\eps_0 \lambda_\nu \tau}\norm{y\tilde\omega}_{L^2}^2 \dd \tau\right)^{\frac{1}{2}} \left(\int_0^{T_\nu}\norm{\nabla\omega}_{L^2}^2 \right)^{\frac{1-\theta}{2}} \left(\int_0^{T_\nu}\norm{D^2\omega}_{L^2}^2 \right)^{\frac{\theta}{2}}\\
  &\qquad \leq C_\theta^2 \nu^{-3/4+\mu\theta}(1+\abs{\log\nu})^{1/2}\norm{\omega^{in}}_{L^2}^{1-\theta}\norm{\tilde\omega^{in}}_{X},
 \end{aligned}
\end{equation}
where in the last step we used the energy estimates \eqref{eq:full_energy} and the bound \eqref{eq:useful_bds} for second order derivatives.
Choosing $\theta=\frac{4\mu}{4\mu+3}$, using \eqref{eq:assumpindat}, and adding the above two estimates implies that
\begin{equation}\label{eq:lfasd3}
  \int_0^{T_\nu} \e^{\eps_0 \lambda_\nu \tau}\norm{y\tilde{u}\cdot\nabla\omega}_{L^2}\dd \tau \leq C\nu^\mu\norm{\tilde\omega^{in}}_{X}.
\end{equation}
Finally, for the term with $y(u_s\de_x\tilde\omega)$ we proceed as follows: By \eqref{eq:useful_bds}, under our assumption on the initial data  \eqref{eq:assumpindat} we have the bound 
\begin{equation}
 \norm{yu_s}_{L^\infty}\leq C \norm{yu_s}_{L^2_y}^{1/2}\norm{yu_s}_{\dot{H}_y^{1}}^{1/2}\leq C \norm{yu_s}_{L^2_y}^{1/2}[\norm{y\omega_s}_{L^2_y}^{1/2}+\norm{u_s}_{L^2_y}^{1/2}]\leq C \nu^{3/4+2\mu}.
\end{equation}
Hence
\begin{equation}\label{eq:lfasd4}
\begin{aligned}
 \int_0^{T_\nu} \e^{\eps_0 \lambda_\nu \tau}\norm{yu_s\de_x\tilde\omega}_{L^2}\dd \tau&\leq \int_0^{T_\nu} \e^{\eps_0 \lambda_\nu \tau}\norm{yu_s}_{L^\infty}\norm{\nabla\tilde\omega}_{L^2}\dd \tau\\
 & \leq C\nu^{3/4+2\mu}\left(\int_0^{T_\nu} \e^{2\eps_0 \lambda_\nu \tau}\dd \tau \right)^{1/2}\nu^{-1/2}\norm{\tilde\omega^{in}}_{L^2}\\
 & \leq C\nu^{1/4+2\mu}\lambda_\nu^{-1/2}\norm{\tilde\omega^{in}}_{L^2}\\
 & = C\nu^{2\mu}(1+\abs{\log\nu})^{1/2} \norm{\tilde\omega^{in}}_{L^2}.
\end{aligned}
\end{equation}
Collecting \eqref{eq:bdutil2}, \eqref{eq:bdus2}, \eqref{eq:lfasd3} and \eqref{eq:lfasd4}, we deduce \eqref{eq:stronger}, therefore concluding the proof.

\end{proof}

\addtocontents{toc}{\protect\setcounter{tocdepth}{-1}}
\section*{Acknowledgements}
T.\ Elgindi acknowledges funding from the NSF DMS-1817134. K.\ Widmayer acknowledges funding from SNSF Grant 157694. The authors thank the departments of mathematics at EPFL, Imperial College and UCSD where a part of this work was completed.  

\addtocontents{toc}{\protect\setcounter{tocdepth}{1}}


\begin{bibdiv}
\begin{biblist}

\bib{BW13}{article}{
      author={Beck, Margaret},
      author={Wayne, C.~Eugene},
       title={Metastability and rapid convergence to quasi-stationary bar
  states for the two-dimensional {N}avier-{S}tokes equations},
        date={2013},
     journal={Proc. Roy. Soc. Edinburgh Sect. A},
      volume={143},
      number={5},
       pages={905\ndash 927},
         url={http://dx.doi.org/10.1017/S0308210511001478},
}

\bib{BGM15I}{article}{
      author={{Bedrossian}, J.},
      author={{Germain}, P.},
      author={{Masmoudi}, N.},
       title={{Dynamics near the subcritical transition of the 3D Couette flow
  I: Below threshold case}},
        date={2015-06},
     journal={ArXiv e-prints, to appear in Mem. Amer. Math. Soc. },
      eprint={1506.03720},
}

\bib{BGM15II}{article}{
      author={{Bedrossian}, J.},
      author={{Germain}, P.},
      author={{Masmoudi}, N.},
       title={{Dynamics near the subcritical transition of the 3D Couette flow
  II: Above threshold case}},
        date={2015-06},
     journal={ArXiv e-prints},
      eprint={1506.03721},
}

\bib{BCZ15}{article}{
      author={Bedrossian, Jacob},
      author={Coti~Zelati, Michele},
       title={Enhanced dissipation, hypoellipticity, and anomalous small noise
  inviscid limits in shear flows},
        date={2017},
     journal={Arch. Ration. Mech. Anal.},
      volume={224},
      number={3},
       pages={1161\ndash 1204},
}

\bib{BCZGH15}{article}{
      author={Bedrossian, Jacob},
      author={Coti~Zelati, Michele},
      author={Glatt-Holtz, Nathan},
       title={Invariant {M}easures for {P}assive {S}calars in the {S}mall
  {N}oise {I}nviscid {L}imit},
        date={2016},
     journal={Comm. Math. Phys.},
      volume={348},
      number={1},
       pages={101\ndash 127},
}

\bib{BGM15III}{article}{
      author={Bedrossian, Jacob},
      author={Germain, Pierre},
      author={Masmoudi, Nader},
       title={On the stability threshold for the 3{D} {C}ouette flow in
  {S}obolev regularity},
        date={2017},
     journal={Ann. of Math. (2)},
      volume={185},
      number={2},
       pages={541\ndash 608},
}

\bib{BGM18}{article}{
      author={Bedrossian, Jacob},
      author={Germain, Pierre},
      author={Masmoudi, Nader},
       title={Stability of the couette flow at high reynolds numbers in two
  dimensions and three dimensions},
        date={2018},
     journal={Bulletin of the American Mathematical Society},
}

\bib{BMV14}{article}{
      author={Bedrossian, Jacob},
      author={Masmoudi, Nader},
      author={Vicol, Vlad},
       title={Enhanced dissipation and inviscid damping in the inviscid limit
  of the {N}avier-{S}tokes equations near the two dimensional {C}ouette flow},
        date={2016},
     journal={Arch. Ration. Mech. Anal.},
      volume={219},
      number={3},
       pages={1087\ndash 1159},
         url={http://dx.doi.org/10.1007/s00205-015-0917-3},
}

\bib{BVW16}{article}{
      author={Bedrossian, Jacob},
      author={Vicol, Vlad},
      author={Wang, Fei},
       title={The {S}obolev stability threshold for 2{D} shear flows near
  {C}ouette},
        date={2018},
     journal={J. Nonlinear Sci.},
      volume={28},
      number={6},
       pages={2051\ndash 2075},
}

\bib{BL94}{article}{
      author={Bernoff, Andrew~J.},
      author={Lingevitch, Joseph~F.},
       title={Rapid relaxation of an axisymmetric vortex},
        date={1994},
     journal={Phys. Fluids},
      volume={6},
      number={11},
       pages={3717\ndash 3723},
}

\bib{CLWZ18}{article}{
      author={{Chen}, Qi},
      author={{Li}, Te},
      author={{Wei}, Dongyi},
      author={{Zhang}, Zhifei},
       title={{Transition threshold for the 2-D Couette flow in a finite
  channel}},
        date={2018-08},
     journal={arXiv e-prints},
      eprint={1808.08736},
}

\bib{CKRZ08}{article}{
      author={Constantin, P.},
      author={Kiselev, A.},
      author={Ryzhik, L.},
      author={Zlatos, A.},
       title={Diffusion and mixing in fluid flow},
        date={2008},
     journal={Ann. of Math. (2)},
      volume={168},
      number={2},
       pages={643\ndash 674},
         url={http://dx.doi.org/10.4007/annals.2008.168.643},
}

\bib{CZDE18}{article}{
      author={{Coti Zelati}, M.},
      author={{Delgadino}, M.G.},
      author={{Elgindi}, T.M.},
       title={{On the relation between enhanced dissipation time-scales and
  mixing rates}},
        date={2018-06},
     journal={ArXiv e-prints, to appear in Comm. Pure Appl. Math.},
      eprint={1806.03258},
}

\bib{DrazinReid81}{book}{
      author={{Drazin}, P.G.},
      author={{Reid}, W.H.},
       title={Hydrodynamic stability},
   publisher={Cambridge University Press, Cambridge},
        date={1981},
}

\bib{Gallay2017}{article}{
      author={{Gallay}, T.},
       title={{Enhanced dissipation and axisymmetrization of two-dimensional
  viscous vortices}},
        date={2017-07},
     journal={ArXiv e-prints},
      eprint={1707.05525},
}

\bib{GNRS18}{article}{
      author={{Grenier}, E.},
      author={{Nguyen}, T.~T.},
      author={{Rousset}, F.},
      author={{Soffer}, A.},
       title={{Linear inviscid damping and enhanced viscous dissipation of
  shear flows by using the conjugate operator method}},
        date={2018-04},
     journal={ArXiv e-prints},
      eprint={1804.08291},
}

\bib{IMM17}{article}{
      author={{Ibrahim}, S.},
      author={{Maekawa}, Y.},
      author={{Masmoudi}, N.},
       title={{On pseudospectral bound for non-selfadjoint operators and its
  application to stability of Kolmogorov flows}},
        date={2017-10},
     journal={ArXiv e-prints},
      eprint={1710.05132},
}

\bib{Kelvin87}{article}{
      author={Kelvin, Lord},
       title={Stability of fluid motion: rectilinear motion of viscous fluid
  between two parallel plates},
        date={1887},
     journal={Phil. Mag.},
      volume={24},
      number={5},
       pages={188\ndash 196},
}

\bib{LB01}{article}{
      author={Latini, Marco},
      author={Bernoff, Andrew~J},
       title={Transient anomalous diffusion in poiseuille flow},
        date={2001},
     journal={J. Fluid Mech.},
      volume={441},
       pages={399\ndash 411},
}

\bib{LiWeiZhang2017}{article}{
      author={{Li}, T.},
      author={{Wei}, D.},
      author={{Zhang}, Z.},
       title={{Pseudospectral and spectral bounds for the Oseen vortices
  operator}},
        date={2017-01},
     journal={ArXiv e-prints},
      eprint={1701.06269},
}

\bib{LWZ18}{article}{
      author={{Li}, Te},
      author={{Wei}, Dongyi},
      author={{Zhang}, Zhifei},
       title={{Pseudospectral bound and transition threshold for the 3D
  Kolmogorov flow}},
        date={2018-01},
     journal={arXiv e-prints},
      eprint={1801.05645},
}

\bib{Liss18}{article}{
      author={{Liss}, Kyle},
       title={{On the Sobolev stability threshold of 3D Couette flow in a
  homogeneous magnetic field}},
        date={2018-12},
     journal={arXiv e-prints},
      eprint={1812.11540},
}

\bib{Lundgren82}{article}{
      author={Lundgren, T.S.},
       title={Strained spiral vortex model for turbulent fine structure},
        date={1982},
     journal={Phys. Fluids},
      volume={25},
      number={12},
       pages={2193\ndash 2203},
}

\bib{rayleigh1879}{article}{
      author={Rayleigh, Lord},
       title={On the stability, or instability, of certain fluid motions},
        date={1879},
     journal={Proceedings of the London Mathematical Society},
      volume={1},
      number={1},
       pages={57\ndash 72},
}

\bib{Reynolds83}{article}{
      author={Reynolds, O.},
       title={An experimental investigation of the circumstances which
  determine whether the motion of water shall be direct or sinuous, and of the
  law of resistance in parallel channels},
        date={1883},
     journal={Proc. R. Soc. Lond.},
      volume={174},
       pages={935\ndash 982},
}

\bib{rhines1983rapidly}{article}{
      author={Rhines, Peter~B.},
      author={Young, William~R.},
       title={How rapidly is a passive scalar mixed within closed
  streamlines?},
        date={1983},
     journal={Journal of Fluid Mechanics},
      volume={133},
       pages={133\ndash 145},
}


\bib{Villani09}{article}{
      author={Villani, C{\'e}dric},
       title={Hypocoercivity},
        date={2009},
     journal={Mem. Amer. Math. Soc.},
      volume={202},
      number={950},
       pages={iv+141},
         url={http://dx.doi.org/10.1090/S0065-9266-09-00567-5},
}

\bib{WZ18}{article}{
      author={{Wei}, D.},
      author={{Zhang}, Z.},
       title={{Transition threshold for the 3D Couette flow in Sobolev space}},
        date={2018-03},
     journal={ArXiv e-prints},
      eprint={1803.01359},
}

\bib{WZZkolmo17}{article}{
      author={{Wei}, D.},
      author={{Zhang}, Z.},
      author={{Zhao}, W.},
       title={{Linear inviscid damping and enhanced dissipation for the
  Kolmogorov flow}},
        date={2017-11},
     journal={ArXiv e-prints},
      eprint={1711.01822},
}

\bib{WEI18}{article}{
      author={{Wei}, Dongyi},
       title={{Diffusion and mixing in fluid flow via the resolvent estimate}},
        date={2018-11},
     journal={arXiv e-prints},
      eprint={1811.11904},
}

\bib{WPKM08}{article}{
      author={Willis, A.~P.},
      author={Peixinho, J.},
      author={Kerswell, R.~R.},
      author={Mullin, T.},
       title={Experimental and theoretical progress in pipe flow transition},
        date={2008},
     journal={Philos. Trans. R. Soc. Lond. Ser. A Math. Phys. Eng. Sci.},
      volume={366},
      number={1876},
       pages={2671\ndash 2684},
}

\end{biblist}
\end{bibdiv}


\end{document}